\documentclass[12pt]{article}
\usepackage{graphicx}
\usepackage{hyperref}
\usepackage{mathrsfs}
\usepackage{amsfonts,amsmath,amsthm, amssymb}
\usepackage{latexsym, euscript, epic, eepic, color}
\usepackage{indentfirst}
\usepackage{color}
\usepackage{enumerate}

\makeatletter
\@addtoreset{equation}{section}
\makeatother
\newtheorem{theorem}{Theorem}[section]
\newtheorem{lemma}[theorem]{Lemma}
\newtheorem{proposition}[theorem]{Proposition}

\newtheorem{corollary}[theorem]{Corollary}

\theoremstyle{definition}
\newtheorem{definition}[theorem]{Definition}

\newtheorem{remark}[theorem]{Remark}
\makeatletter
\newcommand{\rmnum}[1]{\romannumeral #1}
\newcommand{\Rmnum}[1]{\expandafter\@slowromancap\romannumeral #1@}
\makeatother

\begin{document}
\title{Metric results for the eventually always hitting points and level sets in subshift with specification}

\author{Bo Wang and Bing Li\thanks{Corresponding author}\\
\small \it Department of Mathematics, South China University of Technology,\\
\small \it Guangzhou 510640, P.R. China\\
\small \it E-mails:  202010106047@mail.scut.edu.cn,  \\
\small \it           scbingli@scut.edu.cn }
\date{\today}
\maketitle
\begin{center}
\begin{minipage}{120mm}
{\small {\bf Abstract.}
We study the set of eventually always hitting points for symbolic dynamics with specification. The measure and Hausdorff dimension of such sets are obtained.
Moreover, we establish the stronger metric results by introducing a new quantity $L_{N}(\omega)$ which describes the maximal length of string of zeros of the
prefix among the first $N$ iterations of $\omega$ in symbolic
space. The Hausdorff dimensions of the
level sets for this quantity are also completely determined.
 }
\end{minipage}
\end{center}

\section{Introduction}
\subsection{The set of eventually always hitting points}
Let $(X,d,\mathcal{B},\mu,T)$ be a metric measure preserving dynamical system. Namely $\mu$ is a probability measure on $(X,\mathcal{B})$ and $T$-invariant,
$(X,d)$ is a metric space. Here the measure and metric structure are compatible meaning that all Borel sets are $\mu$-measurable.
Let $\{E_{n}\}_{n=1}^{\infty}$ be a sequence of subsets in $\mathcal{B}$ and let
\begin{equation*}
H_{io}(T,\{E_{n}\}):=\{x\in X:T^{n}x\in E_{n}\ {\rm for\ i.o.} \}.
\end{equation*}
Here and throughout, ``i.o.'' stands for ``infinitely often''. The sets $E_{n}$ can be thought of as targets that the orbit under $T$ of points in $X$ have to hit.
The interesting situation is usually, when the diameters of $E_{n}$ tend to zero as $n$ increases.
It is thus natural to refer to $H_{io}(T,\{E_{n}\})$ as the corresponding shrinking target set associated with the given dynamical system and target sets.
The classical ``shrinking target problem'' is to consider the $\mu$-measure and the Hausdorff dimension of $H_{io}(T,\{E_{n}\})$, which was formulated in $\cite{HV1995}$.
There are many results about shrinking target problem, we refer the reader to $\cite{CK2001,FMP2007,HV2002,LLVZ2023}$ for the $\mu$-measure of $H_{io}(T,\{E_{n}\})$ and $\cite{HV1995,HV1997,HV1999,KR2018,LLVZ2023,LWWX2014,SW2013,U2002}$ for the Hausdorff dimension of $H_{io}(T,\{E_{n}\})$.
In this paper we will investigate a set which is closely related to $H_{io}(T,\{E_{n}\})$, namely the set of eventually always hitting points. This is the set of points $x$ satisfies that for all large enough $N$, there exists a $n\in\{1,2,\cdots,N\}$ such that $T^{n}x\in E_{N}$. That is,
\begin{equation*}
\begin{aligned}
H_{ea}(T,\{E_{n}\}):&=\{x\in X:\forall\ N\gg1, 1\leq \exists n\leq N,\ {\rm such\ that}\ T^{n}x\in E_{N}\}\\
&=\bigcup_{m=1}^{\infty}\bigcap_{N=m}^{\infty}\bigcup_{n=1}^{N}T^{-n}E_{N}.
\end{aligned}
\end{equation*}
Here and throughout, ``$N\gg1$'' stands for ``$N$ large enough''. The name eventually always hitting points was introduced by Kelmer $\cite{K2017}$.
We remark that Kelmer $\cite{K2017}$ studied a slightly different definition of eventually always hitting points, and required that for all $N\gg1$,
we have $T^{n}x\in E_{N}$ for some $0\leq n\leq N-1$, whereas we require $1\leq n\leq N$. For the result discussed in this paper,
it is insignificant which definition we use.

Note that $H_{io}(T,\{E_{n}\})$ is a limsup set and $H_{ea}(T,\{E_{n}\})$ is a liminf set. If $\mu(E_{n})\to0$ and $\{E_{n}\}_{n=1}^{\infty}$ is a nested sequence, that is, $E_{n+1}\subset E_{n}$ holds for all $n$, then $H_{ea}(T,\{E_{n}\})\subset H_{io}(T,\{E_{n}\})$ except for a null measure set $\cite{KKP2020}$. In this sense, $H_{ea}(T,\{E_{n}\})$ is smaller than $H_{io}(T,\{E_{n}\})$. Similar with the shrinking target problem, we have the following two problems.\\
(\textbf{P1}) What is the $\mu$-measure of $H_{ea}(T,\{E_{n}\})$?\\
(\textbf{P2}) What is the Hausdorff dimension of $H_{ea}(T,\{E_{n}\})$ if $\mu(H_{ea}(T,\{E_{n}\}))=0$?

There have recently been several results about the $\mu$-measure of the set $H_{ea}(T,\{E_{n}\})$. If $\mu$ is ergodic and $\{E_{n}\}$ is a nested sequence,
the $\mu$-measure of $H_{ea}(T,\{E_{n}\})$ is either zero or one (see $\cite{KKP2020}$).
\begin{enumerate}[\textbullet]
\item Kelmer gave necessary and sufficient conditions for $\mu(H_{ea}(T,\{E_{n}\}))=1$ in the setting of flows on hyperbolic manifolds $\cite{K2017}$, and for flows on homogeneous spaces together with Yu $\cite{KY2019}$.
\item In the general setting, Kelmer showed $\cite{K2017}$ that if there exists $c<1$ such that $\mu(E_{n})\leq\frac{c}{n}$ for infinitely many $n$, then $\mu(H_{ea}(T,\{E_{n}\}))=0$.
\item Kleinbock, Konstantoulas and Richter gave sufficient conditions for $\mu(H_{ea}(T,\{E_{n}\}))$ $=0$ ($\mu(H_{ea}(T,\{E_{n}\}))=1$) respecitively in some mixing systems $\cite{KKR2019}$, such as Gauss map, Product systems and Bernoulli schemes.
For the Gauss map and balls $E_{n}$ with centers in 0, they proved that if there exists $c>\frac{1}{\log 2}$, such that $\mu(E_{n})\geq\frac{c\log\log n}{n}$ for all but finitely many $n$, then $\mu(H_{ea}(T,\{E_{n}\}))=1$, and if there exists $c<\frac{1}{\log 2}$, such that $\mu(E_{n})\leq\frac{c\log\log n}{n}$ for all but finitely many $n$, then $\mu(H_{ea}(T,\{E_{n}\}))=0$.
\item Kirseborn, Kunde and Persson $\cite{KKP2020}$ obtained results similar to those of Kleinbock, Konstantoulas and Richter.
Their results hold for a wider class of dynamical systems, but they need to add stronger conditions on $\mu(E_{n})$ in order to be able to conclude that the measure of $H_{ea}(T,\{E_{n}\})$ is one or zero. For piecewise expanding interval map including the doubling map, Gauss map, and for some quadratic maps, they proved that $\mu(H_{ea}(T,\{E_{n}\}))=1$ when $\mu(E_{n})\geq\frac{c(\log n)^{2}}{n}$ for some $c>0$ and all sufficiently large $n$.
Furthermore, they gave a sufficient condition for $\mu(H_{ea}(T,\{E_{n}\}))=0$ under additional assumptions, they prove that if there exists $c>0$, such that $\mu(E_{n})\leq \frac{c}{n}$ for infinitely many $n$, then $\mu(H_{ea}(T,\{E_{n}\}))=0$.
\item Ganotaki and Persson $\cite{GP2021}$ gave sufficient conditions for $\mu(H_{ea}(T,\{E_{n}\}))=1$ in the systems which have exponential decay of correlations for either H\"{o}lder continuous functions or functions of bounded variation.
They also gave an asymptotic estimate as $N\to\infty$ on the number of $n\leq N$ with $T^{n}x\in E_{N}$.
\item Other recent results can be found in $\cite{KO2021,KW2019}$.
\end{enumerate}
Unfortunately, the sufficient condition and necessary condition above does not coincide.

There are also some results about the Hausdorff dimension of the set $H_{ea}(T,\{E_{n}\})$ for $\beta$-transformation $T_{\beta}(x)=\beta x({\rm mod} \ 1)$ on $X=[0,1]$.
\begin{enumerate}[\textbullet]
\item Bugeaud and Liao $\cite{BL2016}$ proved that $\dim_{\rm H}H_{ea}(T,\{E_{n}\})=\left(\frac{1-v}{1+v}\right)^{2}$ when $E_{n}=[0,\beta^{-vn})$, where $v\in[0,1]$. Here and below, $\dim_{\rm H}$ stands for Hausdorff dimension.
\item Wu $\cite{W2020}$ extended the result of Bugeaud and Liao to every points $x_{0}$ in the unit interval.
\item Wu and Zheng $\cite{WZ2022}$ gave the upper bound and lower bound of $\dim_{\rm H}H_{ea}(T,\{E_{n}\})$ when $E_{n}=[0,\psi(n))$, where $\psi$ is a general positive function.
\end{enumerate}
\subsection{Main results}
In this paper, we will consider the set of eventually always hitting points for symbolic dynamics. Given a positive integer $m\geq2$, let $\mathcal{A}:=\{0,1,\cdots,m-1\}$. We denote
\begin{equation*}
\mathcal{A}^{\mathbb{N}}=\{\omega_{1}\cdots\omega_{n}\cdots: \omega_{i}\in\mathcal{A}, \forall\ i\geq1\}.
\end{equation*}
The set $\mathcal{A}^{\mathbb{N}}$ is compact in the metric $d(u,\omega)=m^{-\min\{i:u_{i}\neq \omega_{i}\}}$ if $u\neq\omega$ and $d(u,\omega)=0$ if $u=\omega$. Furthermore, the shift map $\sigma:\mathcal{A}^{\mathbb{N}}\to\mathcal{A}^{\mathbb{N}}$ defined by $(\sigma u)_{i}=u_{i+1}$ is continuous.
We call $(\mathcal{A}^{\mathbb{N}},d,\sigma)$ the symbolic space. A subshift on $\mathcal{A}$ is a closed set $\Sigma\subset\mathcal{A}^{\mathbb{N}}$ with $\sigma(\Sigma)\subset\Sigma$.
If $\Sigma$ has specification, by the theory of Ruelle operator, there exists a unique $\sigma$-invariant Borel probability measure $\mu$ satisfies the Gibbs property. We refer the reader to Section $\ref{20}$ for more details.

In this paper, we are dedicated to answer (\textbf{P1}) and (\textbf{P2}) when the targets are balls with center $0^{\infty}$ in a subshift with specification.
Assume that $0^{\infty}\in\Sigma$, for function $\psi:\mathbb{R}^{+}\to\mathbb{R}^{+}$ with $\lim\limits_{x\to\infty}\psi(x)=0$,
let $E_{n}=B(0^{\infty},\psi(n))$, where $B(0^{\infty},\psi(n)):=\{\omega\in\Sigma:d(\omega,0^{\infty})<\psi(n)\}$.
For simplicity, we denote $H_{ea}(\sigma,\{E_{n}\})$ by $H_{ea}(\sigma,\psi)$. Denote by $h(\Sigma)$ the topological entropy of $\Sigma$, the definition of topological entropy can be found in Section $\ref{20}$.
Throughout this paper, let $A=e^{h(\Sigma)}$.

We first answer the problems (\textbf{P1}) and (\textbf{P2}) by the two following theorems. Theorem $\ref{19}$ describes the $\mu$-measure of $H_{ea}(\sigma,\psi)$.
Theorem $\ref{110}$ gives the Hausdorff dimension of $H_{ea}(\sigma,\psi)$ when $\lim\limits_{N\to\infty}\frac{-\log_{m}\psi(N)}{N}$ exists.
\begin{theorem}\label{19}
Suppose that $\Sigma$ has specification, $h(\Sigma)>0$ and $0^{\infty}\in\Sigma$. We have
\begin{equation*}
\mu(H_{ea}(\sigma,\psi))=\begin{cases} 0, &{\rm if}\ \limsup\limits_{N\to\infty}\frac{-\log_{m}\psi(N)}{\log_{A}N}>1,\\ 1, &{\rm if}\ \limsup\limits_{N\to\infty}\frac{-\log_{m}\psi(N)}{\log_{A}N}<1.\end{cases}
\end{equation*}
\end{theorem}
\begin{remark}
By Lemma $\ref{24}$, we know that the system $(\Sigma,\sigma,\mu)$ is exponentially mixing, which implies $\sigma$ is ergodic.
Therefore, the first part of Theorem $\ref{19}$ can be deduced from the result of Kirseborn, Kunde and Persson $\cite[\rm Corollary\ 1]{KKP2020}$.
However, their sufficient condition on $\mu(H_{ea})=1$ requires that $X\subset\mathbb{R}$.
\end{remark}
In view of Theorem $\ref{19}$, (\textbf{P2}) can be regarded as
what is the Hausdorff dimension of $H_{ea}(\sigma,\psi)$ if $\limsup\limits_{N\to\infty}\frac{-\log_{m}\psi(N)}{\log_{A}N}>1$? From Corollary
$\ref{121}$ below, we show that $$\dim_{\rm H}H_{ea}(\sigma,\psi)=\dim_{\rm H}\Sigma$$ if $\limsup\limits_{N\to\infty}\frac{-\log_{m}\psi(N)}{\log_{A}N}<+\infty$. The remaining case
is $\limsup\limits_{N\to\infty}\frac{-\log_{m}\psi(N)}{\log_{A}N}=+\infty$, which means that the growth rate of $-\log_{m}\psi(N)$ is larger than $\log N$.
Then the natural candidate for comparing rate is $N$ instead of $\log N$. In the following we assume that $\lim\limits_{N\to\infty}\frac{-\log_{m}\psi(N)}{N}$ exists.
\begin{theorem}\label{110}
Assume that $\Sigma$ has specification, $h(\Sigma)>0$ and $0^{\infty}\in\Sigma$. Suppose $\lim\limits_{N\to\infty}\frac{-\log_{m}\psi(N)}{N}$ exists, denoted by $\tau$
($\tau$ can be $+\infty$).\\ (\rmnum{1}) If $\tau>1$, then
\begin{equation*}
H_{ea}(\sigma,\psi)=\bigcup_{n=0}^{\infty}\sigma^{-n}(\{0^{\infty}\}).
\end{equation*}
Here and below, $\sigma^{-1}(\{\omega\}):=\{u\in\Sigma:\sigma(u)=\omega\}$ for all $\omega\in\Sigma$.\\
(\rmnum{2}) If $0\leq \tau\leq1$, then $$\dim_{\rm H}H_{ea}(\sigma,\psi)=\left(\frac{1-\tau}{1+\tau}\right)^{2}\dim_{\rm H}\Sigma.$$
\end{theorem}
\begin{remark}
It is known that $$\dim_{\rm H}\Sigma=\frac{h(\Sigma)}{\log m}$$ for any subshift $\Sigma$ (for example, see $\rm{\cite[Proposition\ \Rmnum{3}.1]{HF1967}}$).
Thus the dimension formula for Theorem $\ref{110}$ can be reformulated as $$\dim_{\rm H}H_{ea}(\sigma,\psi)=\left(\frac{1-\tau}{1+\tau}\right)^{2}\frac{h(\Sigma)}{\log m}.$$
\end{remark}
By Theorem $\ref{110}$, we immediately obtain the following corollary.
\begin{corollary}\label{121}
Assume that $\Sigma$ has specification, $h(\Sigma)>0$ and $0^{\infty}\in\Sigma$.
For any $\psi$ with $\limsup\limits_{N\to\infty}\frac{-\log_{m}\psi(N)}{\log_{A}N}<+\infty$, we have $\dim_{\rm H}H_{ea}(\sigma,\psi)=\dim_{\rm H}\Sigma$.
\end{corollary}
Moreover, we can establish stronger results, which deduce Theorem $\ref{19}$ and $\ref{110}$ directly due to the quantity $L_{N}(\omega)$ describing
the longest string of zeros. Given $\omega\in\Sigma$, $n\geq1$ and $N\geq1$, define $$l_{n}(\omega)=\sup\{k\geq0:\omega_{n+i}=0,\forall\ 1\leq i\leq k\}$$ and
$$L_{N}(\omega)=\max\limits_{1\leq n\leq N}l_{n}(\omega).$$ The quantity $l_{n}(\omega)$ is the maximal length of string of zeros of the prefix of $\sigma^{n}\omega$,
which was introduced by Li and Wu $\cite{LW2008}$. Li, Persson, Wang and Wu $\cite{LPWW2014}$ gave a kind of classification of $\beta>1$ according to $\{l_{n}\}_{n\geq1}$
and calculated the Lebesgue measure and Hausdorff dimension of each class. Fang, Wu and Li $\cite{FWL2020}$ studied the approximation orders of real numbers by
$\beta$-expansion in terms of investigating the behaviour of $l_{n}(\omega)$. We introduce a new quantity $L_{N}(\omega)$ describing the maximal length of
string of zeros of the prefix among $\{\sigma\omega,\sigma^{2}\omega,\cdots,\sigma^{N}\omega\}$. Recall that
\begin{equation*}
\begin{aligned}
H_{ea}(\sigma,\psi)&=\{\omega\in\Sigma:\forall N\gg1, 1\leq\exists n\leq N,\ {\rm such\ that}\ \sigma^{n}\omega\in E_{N}\}\\
&=\left\{\omega\in\Sigma:\forall N\gg1, \min_{1\leq n\leq N}d(\sigma^{n}\omega,0^{\infty})<\psi(N)\right\}\\
&=\left\{\omega\in\Sigma:\forall N\gg1, L_{N}(\omega)>-\log_{m}\psi(N)-1\right\}.
\end{aligned}
\end{equation*}
We have
\begin{equation}\label{18}
\left\{\omega\in\Sigma:\liminf_{N\to\infty}\frac{L_{N}(\omega)}{-\log_{m}\psi(N)}>1\right\}\subset H_{ea}(\sigma,\psi)\subset\left\{\omega\in\Sigma:\liminf_{N\to\infty}\frac{L_{N}(\omega)}{-\log_{m}\psi(N)}\geq1\right\}.
\end{equation}
So it suffices to consider $\liminf\limits_{N\to\infty}\frac{L_{N}(\omega)}{-\log_{m}\psi(N)}$ for the study of the set $H_{ea}(\sigma,\psi)$. The typical result for the limit of $L_{N}(\omega)$ is the following.
\begin{theorem}\label{16}
Suppose that $\Sigma$ has specification, $h(\Sigma)>0$ and $0^{\infty}\in\Sigma$. For $\mu$-almost all $\omega\in\Sigma$, we have
\begin{equation*}
\lim_{N\to\infty}\frac{L_{N}(\omega)}{\log_{A}N}=1.
\end{equation*}
\end{theorem}
Theorem $\ref{16}$ tells us that for $\mu$-almost all $\omega\in\Sigma$, $L_{N}(\omega)$ increases to infinity with the speed $\log_{A}N$ as $N$ goes to infinity.
Generally, for any $a\geq0$, let
\begin{equation*}
\mathcal{U}_{\psi}(a):=\left\{\omega\in\Sigma:\liminf_{N\to\infty}\frac{L_{N}(\omega)}{-\log_{m}\psi(N)}\geq a\right\}.
\end{equation*}
By $\eqref{18}$, the study of the Hausdorff dimension of $H_{ea}(\sigma,\psi)$ can be transformed into the the study of the Hausdorff dimension of $\mathcal{U}_{\psi}(a)$.
We choose $\psi_{0}$ with $\psi_{0}(N)=m^{-N}$, that is, $-\log_{m}\psi_{0}(N)=N$. So
\begin{equation*}
\mathcal{U}_{\psi_{0}}(a)=\left\{\omega\in\Sigma:\liminf_{N\to\infty}\frac{L_{N}(\omega)}{N}\geq a\right\}.
\end{equation*}
The Hausdorff dimension of $\mathcal{U}_{\psi_{0}}(a)$ is given in Corollary $\ref{17}$ below. Indeed, we establish the stronger results. That is, let $0\leq a\leq b$, we can further study the level sets
\begin{equation*}
\mathcal{U}_{\psi}(a,b):=\left\{\omega\in\Sigma:\liminf_{N\to\infty}\frac{L_{N}(\omega)}{-\log_{m}\psi(N)}\geq a,\limsup_{N\to\infty}\frac{L_{N}(\omega)}{-\log_{m}\psi(N)}=b\right\}
\end{equation*}
and
\begin{equation*}
\mathcal{U}^{\ast}_{\psi}(a,b):=\left\{\omega\in\Sigma:\liminf_{N\to\infty}\frac{L_{N}(\omega)}{-\log_{m}\psi(N)}=a ,\limsup_{N\to\infty}\frac{L_{N}(\omega)}{-\log_{m}\psi(N)}=b\right\}.
\end{equation*}
Obviously, $\mathcal{U}^{\ast}_{\psi}(a,b)\subset\mathcal{U}_{\psi}(a,b)$. Theorem $\ref{113}$ describes the Hausdorff dimensions of $\mathcal{U}_{\psi}(a,b)$ and $\mathcal{U}^{\ast}_{\psi}(a,b)$ completely when $\lim\limits_{N\to\infty}\frac{-\log_{m}\psi(N)}{N}$ exists.
\begin{theorem}\label{113}
Assume that $\Sigma$ has specification, $h(\Sigma)>0$ and $0^{\infty}\in\Sigma$. Suppose $\lim\limits_{N\to\infty}\frac{-\log_{m}\psi(N)}{N}$ exists,
denoted by $\tau$ ($\tau$ can be $+\infty$). \\
Case \rmnum{1}\quad $0<\tau<+\infty$ \\
(1) $a\geq\tau^{-1}$
\begin{enumerate}[\textbullet]
\item If $b<+\infty$, then $$\mathcal{U}^{\ast}_{\psi}(a,b)=\mathcal{U}_{\psi}(a,b)=\emptyset.$$
\item If $b=+\infty$, then $$\dim_{\rm H}\mathcal{U}^{\ast}_{\psi}(a,b)=\dim_{\rm H}\mathcal{U}_{\psi}(a,b)=0.$$
\end{enumerate}
(2) $0\leq a<\tau^{-1}$
\begin{enumerate}[\textbullet]
\item If $b<\frac{a}{1-\tau a}$, then $$\mathcal{U}^{\ast}_{\psi}(a,b)=\mathcal{U}_{\psi}(a,b)=\emptyset.$$
\item If $b\geq\frac{a}{1-\tau a}$, then
\begin{equation*}
\dim_{\rm H}\mathcal{U}^{\ast}_{\psi}(a,b)=\dim_{\rm H}\mathcal{U}_{\psi}(a,b)=\frac{b(1-\tau a)-a}{(1+\tau b)(b-a)}\dim_{\rm H}\Sigma.
\end{equation*}
\end{enumerate}
Case \rmnum{2}\quad  $\tau=+\infty$
\begin{enumerate}[\textbullet]
\item If $b>0$, then $$\dim_{\rm H}\mathcal{U}^{\ast}_{\psi}(a,b)=\dim_{\rm H}\mathcal{U}_{\psi}(a,b)=0.$$
\item If $b=0$, then $$\dim_{\rm H}\mathcal{U}^{\ast}_{\psi}(a,b)=\dim_{\rm H}\mathcal{U}_{\psi}(a,b)=\dim_{\rm H}\Sigma.$$
\end{enumerate}
Case \rmnum{3}\quad  $\tau=0$\\ Additionally, suppose $\psi(x)$ is strictly decreasing and $\lim\limits_{N\to\infty}\frac{\log_{m}\psi(N-\log_{m}\psi(N))}{\log_{m}\psi(N)}=1$, then $$\dim_{\rm H}\mathcal{U}^{\ast}_{\psi}(a,b)=\dim_{\rm H}\mathcal{U}_{\psi}(a,b)=\dim_{\rm H}\Sigma.$$
\end{theorem}
\begin{remark}
(\rmnum{1}) In (2) of Theorem $\ref{113}$ Case \rmnum{1}, we have $$\dim_{\rm H}\mathcal{U}^{\ast}_{\psi}(a,b)=\dim_{\rm H}\mathcal{U}_{\psi}(a,b)=\frac{b(1-\tau a)-a}{(1+\tau b)(b-a)}\dim_{\rm H}\Sigma$$
for all $b\geq\frac{a}{1-\tau a}$. Note that $b(1-\tau a)-a=(1+\tau b)(b-a)=0$ when $a=b=0$. However, we know that the denominator can not be zero. In fact, in the proof of Theorem $\ref{113}$, we show that
$$\dim_{\rm H}\mathcal{U}^{\ast}_{\psi}(0,0)=\dim_{\rm H}\mathcal{U}_{\psi}(0,0)=\dim_{\rm H}\Sigma.$$ Thus, we stipulate that $\frac{b(1-\tau a)-a}{(1+\tau b)(b-a)}=1$ when $a=b=0$.\\
(\rmnum{2}) More precisely, in (1) of Theorem $\ref{113}$ Case \rmnum{1}, we have
\begin{enumerate}[\textbullet]
\item $a>\tau^{-1}$ and $b=+\infty$ $\Rightarrow$ $\mathcal{U}_{\psi}(a,b)=\bigcup_{n=0}^{\infty}\sigma^{-n}(\{0^{\infty}\})$.
\item $\tau^{-1}<a<+\infty$ $\Rightarrow$ $\mathcal{U}^{\ast}_{\psi}(a,b)=\emptyset$.
\end{enumerate}
In Theorem $\ref{113}$ Case \rmnum{2}, we have
\begin{enumerate}[\textbullet]
\item $a>0$ $\Rightarrow$ $\mathcal{U}_{\psi}(a,b)\subset\bigcup_{n=0}^{\infty}\sigma^{-n}(\{0^{\infty}\})$.
\end{enumerate}
\end{remark}
\begin{corollary}\label{17}
Suppose that $\Sigma$ has specification, $h(\Sigma)>0$ and $0^{\infty}\in\Sigma$. We have\\
(\rmnum{1}) If $a>1$, then
\begin{equation*}
\mathcal{U}_{\psi_{0}}(a)=\bigcup_{n=0}^{\infty}\sigma^{-n}(\{0^{\infty}\}).
\end{equation*}
(\rmnum{2}) If $0\leq a\leq1$, then
\begin{equation}\label{111}
\dim_{\rm H}\mathcal{U}_{\psi_{0}}(a)=\left(\frac{1-a}{1+a}\right)^{2}\dim_{\rm H}\Sigma.
\end{equation}
\end{corollary}
The remainder of the paper is organized as follows. In Section $\ref{10}$,
we give the proof of Theorem $\ref{19}$ (Theorem $\ref{110}$) modulo Theorem $\ref{16}$ (Corollary $\ref{17}$) respectively. In Section $\ref{20}$, we introduce some preliminary concepts, including symbolic dynamics and Ruelle operator.
The proof of Theorem $\ref{16}$ is given in Section $\ref{30}$. In Section $\ref{40}$, we calculate the Hausdorff dimensions
of $\mathcal{U}_{\psi_{0}}(a,b)$ and $\mathcal{U}_{\psi_{0}}^{\ast}(a,b)$. Section $\ref{50}$ is dedicated to the proof
of Theorem $\ref{113}$ and Corollary $\ref{17}$.
\section{Establishing Theorems $\ref{19}$ and $\ref{110}$}\label{10}
\begin{proof}[\rm \textbf{Proof of Theorem $\ref{19}$ modulo Theorem $\ref{16}$}]
By $\eqref{18}$, it suffices to show that
\begin{equation}\label{114}
\mu\left(\left\{\omega\in\Sigma:\liminf_{N\to\infty}\frac{L_{N}(\omega)}{-\log_{m}\psi(N)}\geq1\right\}\right)=0\ {\rm if}\ \limsup\limits_{N\to\infty}\frac{-\log_{m}\psi(N)}{\log_{A}N}>1
\end{equation}
and
\begin{equation}\label{115}
\mu\left(\left\{\omega\in\Sigma:\liminf_{N\to\infty}\frac{L_{N}(\omega)}{-\log_{m}\psi(N)}>1\right\}\right)=1\ {\rm if}\ \limsup\limits_{N\to\infty}\frac{-\log_{m}\psi(N)}{\log_{A}N}<1.
\end{equation}
Denote $\Gamma=\left\{\omega\in\Sigma:\lim\limits_{N\to\infty}\frac{L_{N}(\omega)}{\log_{A}N}=1\right\}$. Note that for any $\omega\in\Gamma$, we have
\begin{equation}\label{116}
\begin{aligned}
\liminf_{N\to\infty}\frac{L_{N}(\omega)}{-\log_{m}\psi(N)}&=\liminf_{N\to\infty}\frac{L_{N}(\omega)}{\log_{A}N}\cdot\frac{\log_{A}N}{-\log_{m}\psi(N)}\\
&=\liminf_{N\to\infty}\frac{\log_{A}N}{-\log_{m}\psi(N)}\\ &=\left(\limsup\limits_{N\to\infty}\frac{-\log_{m}\psi(N)}{\log_{A}N}\right)^{-1}.
\end{aligned}
\end{equation}
The combination of Theorem $\ref{16}$ and $\eqref{116}$ gives $\eqref{114}$ and $\eqref{115}$.
\textrm{}
\end{proof}

\begin{proof}[\rm \textbf{Proof of Theorem $\ref{110}$ modulo Corollary $\ref{17}$}]
We use two important facts. \\ Fact 1: When $\tau>0$, for any $\omega\in\mathcal{A}^{\mathbb{N}}$ with $\liminf\limits_{N\to\infty}\frac{L_{N}(\omega)}{-\log_{m}\psi(N)}>0$, we have
\begin{equation}\label{117}
\liminf_{N\to\infty}\frac{L_{N}(\omega)}{N}=\liminf_{N\to\infty}\frac{L_{N}(\omega)}{-\log_{m}\psi(N)}\cdot\frac{-\log_{m}\psi(N)}{N}=\tau\liminf_{N\to\infty}\frac{L_{N}(\omega)}{-\log_{m}\psi(N)}
\end{equation}
Furthermore, $\eqref{117}$ holds for $0<\tau<+\infty$ and any $\omega\in\mathcal{A}^{\mathbb{N}}$.\\
Fact 2: When $\tau=0$, for any $\omega\in\mathcal{A}^{\mathbb{N}}$ with $\liminf\limits_{N\to\infty}\frac{L_{N}(\omega)}{N}>0$, we have
\begin{equation}\label{119}
\liminf\limits_{N\to\infty}\frac{L_{N}(\omega)}{-\log_{m}\psi(N)}=\liminf_{N\to\infty}\frac{L_{N}(\omega)}{N}\cdot\frac{N}{-\log_{m}\psi(N)}=+\infty.
\end{equation}
(\rmnum{1}) If $\tau>1$, by Corollary $\ref{17}$ (\rmnum{1}), we have
\begin{equation}\label{118}
\mathcal{U}_{\psi_{0}}(\tau)=\bigcup_{n=0}^{\infty}\sigma^{-n}(\{0^{\infty}\}).
\end{equation}
The upshot of $\eqref{18}$, $\eqref{117}$ and $\eqref{118}$ is that $H_{ea}(\sigma,\psi)\subset\mathcal{U}_{\psi_{0}}(\tau)=\bigcup_{n=0}^{\infty}\sigma^{-n}(\{0^{\infty}\})$.
On the other hand, for any $\omega\in\bigcup_{n=0}^{\infty}\sigma^{-n}(\{0^{\infty}\})$, there exists $n_{0}\geq0$, such that $\sigma^{n_{0}}\omega=0^{\infty}$.
Thus, $L_{N}(\omega)=+\infty$ for every $N\geq n_{0}$, which implies that $\omega\in H_{ea}(\sigma,\psi)$. Thus,
$\bigcup_{n=0}^{\infty}\sigma^{-n}(\{0^{\infty}\})\subset H_{ea}(\sigma,\psi)$. Therefore,
\begin{equation*}
H_{ea}(\sigma,\psi)=\bigcup_{n=0}^{\infty}\sigma^{-n}(\{0^{\infty}\}).
\end{equation*}
(\rmnum{2}) If $0<\tau\leq1$, the upshot of $\eqref{18}$ and $\eqref{117}$ is that
\begin{equation}\label{120}
\mathcal{U}_{\psi_{0}}(\tau(1+\epsilon))\subset H_{ea}(\sigma,\psi)\subset\mathcal{U}_{\psi_{0}}(\tau)
\end{equation}
for any $\epsilon>0$. When $\tau=1$, in view of $\eqref{120}$ and Corollary $\ref{17}$ (\rmnum{2}), we obtain that $\dim_{\rm H}H_{ea}(\sigma,\psi)=0$.
When $0<\tau<1$, for any $\epsilon\in(0,\tau^{-1}-1)$, the combination of $\eqref{120}$ and Corollary $\ref{17}$ (\rmnum{2}) gives
\begin{equation*}
\left(\frac{1-\tau(1+\epsilon)}{1+\tau(1+\epsilon)}\right)^{2}\dim_{\rm H}\Sigma\leq\dim_{\rm H}H_{ea} (\sigma,\psi)\leq\left(\frac{1-\tau}{1+\tau}\right)^{2}\dim_{\rm H}\Sigma.
\end{equation*}
Letting $\epsilon\to0$, we have $$\dim_{\rm H}H_{ea}(\sigma,\psi)=\left(\frac{1-\tau}{1+\tau}\right)^{2}\dim_{\rm H}\Sigma.$$
The remaining case is $\tau=0$, by $\eqref{18}$ and $\eqref{118}$, we have $\mathcal{U}_{\psi_{0}}(a)\subset H_{ea}(\sigma,\psi)$ for all $0<a<1$.
It follows from Corollary $\ref{17}$ (\rmnum{2}) that $\dim_{\rm H}H_{ea}(\sigma,\psi)\geq\left(\frac{1-a}{1+a}\right)^{2}\dim_{\rm H}\Sigma$ for all $0<a<1$,
which implies that $\dim_{\rm H}H_{ea}(\sigma,\psi)\geq\dim_{\rm H}\Sigma$. Thus $$\dim_{\rm H}H_{ea}(\sigma,\psi)=\dim_{\rm H}\Sigma.$$
\textrm{}
\end{proof}
\section{Preliminary}\label{20}
\subsection{Symbolic dynamics}
Let $(\mathcal{A}^{\mathbb{N}},d,\sigma)$ be the symbolic space. For any $n\geq1$, define
\begin{equation*}
\mathcal{A}^{n}=\{\omega_{1}\cdots\omega_{n}:\omega_{i}\in\mathcal{A}, \forall\ 1\leq i\leq n \}.
\end{equation*}
The element in $\mathcal{A}^{\mathbb{N}}$ and $\mathcal{A}^{n}$ is called the infinite word and the word of length $n$, respectively.
We also denote $\mathcal{A}^{\ast}$ by the collection of all finite words, i.e. $\mathcal{A}^{\ast}=\cup_{n\geq1}\mathcal{A}^{n}$.
For any $\omega,u\in\mathcal{A}^{\ast}$, we write $\omega u$ as the concatenation of words $\omega$ and $u$.
In particular, $\omega^{n}$ denotes the $n$ times self-concatenation of $\omega$ and $\omega^{\infty}$ stands for the infinite times self-concatenation of $\omega$.
The lexicographical order $\prec$ on $\mathcal{A}^{\mathbb{N}}$ is defined as follows:
\begin{equation*}
\varepsilon_{1}\varepsilon_{2}\cdots\varepsilon_{n}\cdots\prec\omega_{1}\omega_{2}\cdots\omega_{n}\cdots
\end{equation*}
if there exists an integer $n\geq1$ such that $\varepsilon_{k}=\omega_{k}$ for all $1\leq k<n$ but $\varepsilon_{n}<\omega_{n}$.
Similarly, for each $n\in\mathbb{N}$, the lexicographical order $\prec$ on $\mathcal{A}^{n}$ is defined as follows:
\begin{equation*}
\varepsilon_{1}\varepsilon_{2}\cdots\varepsilon_{n}\prec\omega_{1}\omega_{2}\cdots\omega_{n}
\end{equation*}
if there exists an integer $i\leq n$ such that $\varepsilon_{k}=\omega_{k}$ for all $1\leq k<i$ but $\varepsilon_{i}<\omega_{i}$.
Given $u\in\mathcal{A}^{\mathbb{N}}$ and positive integers $i\leq j$, let $u_{[i,j]}=u_{i}u_{i+1}\cdots u_{j}$. Let $\Sigma\subset\mathcal{A}^{\mathbb{N}}$ be a subshift. For each $n\geq1$, define
\begin{equation*}
\Sigma^{n}=\{\omega_{1}\cdots\omega_{n}\in\mathcal{A}^{n}: \exists\ u\in\Sigma, {\rm such\ that}\ u_{[1,n]}=\omega_{1}\cdots\omega_{n}\}
\end{equation*}
which is the set of all finite word of length $n$ appears in some elements of $\Sigma$.
Similar with $\mathcal{A}^{\ast}$, we denote $\Sigma^{\ast}$ the collection of all finite words in $\Sigma$, that is, $\Sigma^{\ast}=\cup_{n\geq1}\Sigma^{n}$.
For $u\in\Sigma\cup\Sigma^{\ast}$ and $n\in\mathbb{N}$, denote $I_{n}(u)$ the set of all $\omega\in\Sigma$ such that $\omega_{[1,n]}=u_{1}\cdots u_{n}$. We call $I_{n}(u)$ the $n$-th cylinder containing $u$. Furthermore, for any $\omega\in\Sigma$ and $r>0$, we denote $B(\omega,r):=\{u\in\Sigma:d(u,\omega)<r\}$, which is the ball with center $\omega$ and radius $r$. Let $\sharp$ denote the cardinality of a finite set. Note that
\begin{equation}\label{11}
\sharp\Sigma^{i+j}\leq\sharp\Sigma^{i}\cdot\sharp\Sigma^{j}
\end{equation}
for all $i,j\in\mathbb{N}$, which is a simple but important property.

We introduce two important definitions, which appears many times in our paper. The first definition is the topological entropy, which can be found in $\cite[\rm Chapter\ 7]{PW1982}$.
\begin{definition}
The topological entropy of a subshift $\Sigma$ is
\begin{equation}\label{12}
h(\Sigma):=\lim_{n\to\infty}\frac{\log\sharp\Sigma^{n}}{n}=\inf_{n\to\infty}\frac{\log\sharp\Sigma^{n}}{n},
\end{equation}
where existence of the limit and the second equality follow from $\eqref{11}$ and a standard lemma regarding subadditive sequences.
\end{definition}
In fact, the above definition is not the original one $\cite[\rm Definition\ 7.6]{PW1982}$, but in view of $\cite[\rm Theorem\ 7.13]{PW1982}$, for subshift, the above definition is equivalent to the original one. The below definition of specification is specialised for the symbolic setting and is slightly weaker than Bowen's original one $\cite{RB1974}$.
\begin{definition}
A subshift $\Sigma$ has the specification property or simply has specification if there exists $M\in\mathbb{N}$ such that for any $u,v\in\Sigma^{\ast}$, there exists $\omega\in\Sigma^{M}$ such that $u\omega v\in\Sigma^{\ast}$.
\end{definition}
There are many examples of symbolic dynamics with specification, we refer the reader to $\cite{BRA2014,BM1986}$.

\subsection{Ruelle operator}
An important tool to find $\sigma$-invariant measure on $\Sigma$ is Ruelle operator.
Denote by $\mathcal{C}=\mathcal{C}(\Sigma,\mathbb{R})$ the space of all continuous functions $\phi:\Sigma\to\mathbb{R}$ with the supremum norm $\|\phi\|:=\max\limits_{\omega\in\Sigma}|\phi(\omega)|$.
We define a linear operator $\mathcal{L}$ as
\begin{equation}\label{112}
\mathcal{L}\phi(\omega):=\sum_{u:\sigma u=\omega}\phi(u),\quad \phi\in\mathcal{C}.
\end{equation}
The above operator is called Ruelle operator. Here we choose a special potential, that is, $\varphi(\omega)=1$, $\forall\ \omega\in\Sigma$.
Note that $d(\sigma u,\sigma v)=m\cdot d(u,v)$ for any $u,v\in\Sigma$ with $d(u,v)\leq m^{-2}$.
In view of $\cite[\rm Proposition\ 1]{FJ2001}$, there exists $0<\kappa\leq m^{-2}$, such that $\sharp\sigma^{-1}(u)=\sharp\sigma^{-1}(v)$ for all $u,v\in\Sigma$ with $d(u,v)\leq\kappa$.
Furthermore, we can arrange $\sigma^{-1}(u)=\{x_{1},\cdots,x_{n}\}$ and $\sigma^{-1}(v)=\{y_{1},\cdots,y_{n}\}$ such that $d(x_{i},y_{i})\leq m^{-1}\cdot d(u,v)$ for all $1\leq i\leq n$. Hence, we have $\mathcal{L}\phi\in\mathcal{C}$ for every $\phi\in\mathcal{C}$. Let $\mathcal{M}$ be the dual space of $\mathcal{C}$, by the Riesz representation theorem, $\mathcal{M}$ is actually the space of all Borel measures on $\Sigma$. Let $\mathcal{L}^{\ast}:\mathcal{M}\to\mathcal{M}$ be the dual operator of $\mathcal{L}:\mathcal{C}\to\mathcal{C}$. That is, for any $\phi\in\mathcal{C}$ and any $\nu\in\mathcal{M}$,
\begin{equation*}
\int\phi\mathrm{d}\mathcal{L}^{\ast}\nu=\int \mathcal{L}\phi\mathrm{d}\nu.
\end{equation*}
For simplicity, we denote $\int\phi\mathrm{d}\nu$ by $\langle\nu,\phi\rangle$ for each $\phi\in\mathcal{C}$ and $\nu\in\mathcal{M}$. Due to the result of Fan and Jiang \cite[\rm Theorem 1]{FJ2001}, there are a strictly positive number $\rho$ and a strictly positive function $h$ such that $\mathcal{L}h=\rho h$. Furthermore, there exists a probability measure $\nu$ satsifying $\mathcal{L}^{\ast}\nu=\rho\nu$. Choose $h$ such that $\langle\nu,h\rangle=1$. Let $\mu:=h\nu$, i.e. $\frac{\mathrm{d}\mu}{\mathrm{d}\nu}=h$. The measure $\mu$ has the Gibbs property: there exists a constant $\gamma>1$ such that
\begin{equation}\label{15}
\gamma^{-1}\rho^{-n}\leq\mu(I_{n}(\omega))\leq\gamma\rho^{-n}
\end{equation}
holds for all $\omega\in\Sigma$ and $n\geq1$. What is more, by the result of Fan and Jiang \cite[\rm Theorem 1]{FJ2001}, $\mu$ is the unique $\sigma$-invariant Borel probability measure which satisfies the Gibbs property. Take $\widetilde{\varphi}$ as
\begin{equation*}
\widetilde{\varphi}(\omega)=\frac{h(\omega)}{\rho\cdot h(\sigma\omega)}>0,\ \forall\ \omega\in\Sigma.
\end{equation*}
Consider the Ruelle operator associated to $\widetilde{\mathcal{L}}=\mathcal{L}_{\widetilde{\varphi}}$. The important feature of $\widetilde{\mathcal{L}}$ is that
\begin{equation}\label{13}
\widetilde{\mathcal{L}}1=1.
\end{equation}
We call it the normalized Ruelle operator. Denote $\widetilde{\mathcal{L}}^{\ast}=\mathcal{L}_{\widetilde{\varphi}}^{\ast}$, which is the dual operator of $\widetilde{\mathcal{L}}$. We have the following relation between $\mathcal{L}^{\ast}$ and $\widetilde{\mathcal{L}}^{\ast}$:
\begin{equation*}
\mathcal{L}^{\ast n}\nu=\rho^{n}h^{-1}\widetilde{\mathcal{L}}^{\ast n}(h\nu),\ \forall n\geq1.
\end{equation*}
Thus, we have that $\widetilde{\mathcal{L}}^{\ast}\mu=\mu$.

\section{Proof of Theorem $\ref{16}$}\label{30}
The following lemma implies that the eigenvalue $\rho$ is equal to $e^{h(\Sigma)}$ under the potential $\varphi$ with $\varphi(\omega)=1$, $\forall\ \omega\in\Sigma$.
\begin{lemma}\label{21}
Assume that $\Sigma$ has specification and $h(\Sigma)>0$. Then for any $\omega\in\Sigma$ and all $n\in\mathbb{N}$, we have
\begin{equation}\label{31}
\gamma^{-1}e^{-n\cdot h(\Sigma)}\leq\mu(I_{n}(\omega))\leq\gamma e^{-n\cdot h(\Sigma)},
\end{equation}
where $\gamma>1$ is the constant that appears in $\eqref{15}$.
\end{lemma}
\begin{proof}[\rm \textbf{Proof}]
The proof is standard, we give the proof here for completeness. By $\eqref{15}$, we only need to show that $\rho=e^{h(\Sigma)}$. In view of \cite[\rm Theorem 1]{FJ2001}, we know that
\begin{equation}\label{22}
\lim_{n\to\infty}\|\mathcal{L}^{n}1\|^{\frac{1}{n}}=\rho.
\end{equation}
Then we show that $\lim\limits_{n\to\infty}\|\mathcal{L}^{n}1\|^{\frac{1}{n}}=e^{h(\Sigma)}$. On the one hand, since
\begin{equation*}
\mathcal{L}^{n}1(\omega)=\sum_{u:\sigma^{n}u=\omega}1\leq\sharp\Sigma^{n}
\end{equation*}
for every $\omega\in\Sigma$, we have $\|\mathcal{L}^{n}1\|\leq\sharp\Sigma^{n}$. On the other hand, in view of $\Sigma$ has specification, for any $n>M$ and any $v\in\Sigma^{n-M}$, there exists $u\in\Sigma^{M}$, such that $vu\omega\in\Sigma$, which implies that $vu\omega\in\sigma^{-n}\omega$. Therefore,
\begin{equation*}
\sum_{u:\sigma^{n}u=\omega}1\geq\sharp\Sigma^{n-M}.
\end{equation*}
Hence, $\|\mathcal{L}^{n}1\|\geq\sharp\Sigma^{n-M}$. Therefore
\begin{equation}\label{23}
\lim_{n\to\infty}\|\mathcal{L}^{n}1\|^{\frac{1}{n}}=\lim_{n\to\infty}(\sharp\Sigma^{n})^{\frac{1}{n}}=e^{h(\Sigma)}.
\end{equation}
The upshot of $\eqref{22}$ and $\eqref{23}$ is that $\rho=e^{h(\Sigma)}$.
\textrm{}
\end{proof}
The following statement shows that the system $(\Sigma,\sigma,\mu)$ is exponentially mixing when $\Sigma$ has specification.
\begin{lemma}\label{24}
Suppose that $\Sigma$ has specification. The system $(\Sigma,\sigma,\mu)$ is exponentially mixing, i.e. there exists $0<\theta<1$, such that for any cylinder $E$ and any measurable set $F$, we have
\begin{equation*}
\mu(E\cap\sigma^{-n}F)=\mu(E)\mu(F)+\mu(F)\cdot O(\theta^{n}),
\end{equation*}
for all $n\geq1$, where the constant involved by the $O$ is absolute.
\end{lemma}
\begin{proof}[\rm \textbf{Proof}]
Given $\alpha>0$, denote by $\mathcal{C}^{\alpha}$ the set of all $\alpha$-H\"{o}lder continuous function. We prove that there exists $c>0$ and $\theta\in(0,1)$, such that for any $f\in L^{1}(\Sigma,\mu)$ and $g\in\mathcal{C}^{\alpha}$, we have
\begin{equation*}
\left|\int f\circ\sigma^{n}\cdot g\mathrm{d}\mu-\int f\mathrm{d}\mu\cdot\int g\mathrm{d}\mu\right|\leq c\cdot\int|f|\mathrm{d}\mu\cdot\theta^{n}.
\end{equation*}
In view of $\widetilde{\mathcal{L}}^{\ast}\mu=\mu$ and $\langle\widetilde{\mathcal{L}}^{\ast}\mu,f\rangle=\langle\mu,\widetilde{\mathcal{L}}f\rangle$, we have that
\begin{equation*}
\begin{aligned}
\int f\circ\sigma^{n}\cdot g\mathrm{d}\mu-\int f\mathrm{d}\mu\cdot\int g\mathrm{d}\mu&=\int f\circ\sigma^{n}\cdot g\mathrm{d}((\widetilde{\mathcal{L}}^{\ast})^{n}\mu)-\int f\mathrm{d}\mu\cdot\int g\mathrm{d}\mu\\
&=\int \widetilde{\mathcal{L}}^{n}(f\circ\sigma^{n}\cdot g)\mathrm{d}\mu-\int f\mathrm{d}\mu\cdot\int g\mathrm{d}\mu.
\end{aligned}
\end{equation*}
Note that $\widetilde{\mathcal{L}}^{n}(f\circ\sigma^{n}\cdot g)=f\cdot\widetilde{\mathcal{L}}^{n}g$. Therefore,
\begin{equation*}
\begin{aligned}
\int f\circ\sigma^{n}\cdot g\mathrm{d}\mu-\int f\mathrm{d}\mu\cdot\int g\mathrm{d}\mu&=\int f\cdot\widetilde{\mathcal{L}}^{n}g\mathrm{d}\mu-\int f\mathrm{d}\mu\cdot\int g\mathrm{d}\mu\\&=\int f\cdot(\widetilde{\mathcal{L}}^{n}g-\langle\mu,g\rangle)\mathrm{d}\mu.
\end{aligned}
\end{equation*}
Thus,
\begin{equation*}
\begin{aligned}
\left|\int f\circ\sigma^{n}\cdot g\mathrm{d}\mu-\int f\mathrm{d}\mu\cdot\int g\mathrm{d}\mu\right|&=\left|\int f\cdot(\widetilde{\mathcal{L}}^{n}g-\langle\mu,g\rangle)\mathrm{d}\mu\right|\\
&\leq\int |f|\mathrm{d}\mu\cdot \|\widetilde{\mathcal{L}}^{n}g-\langle\mu,g\rangle\|\\ &=\int |f|\mathrm{d}\mu\cdot \|\widetilde{\mathcal{L}}^{n}g-\langle\mu,g\rangle\cdot1\|.
\end{aligned}
\end{equation*}
Due to $\eqref{13}$, 1 is the strictly positive eigenfunction of $\widetilde{\mathcal{L}}$.
By $\cite[\rm Theorem\ 2.2]{PP1990}$, there exists $c>0$ and $\theta\in(0,1)$ such that for any $g\in\mathcal{C}^{\alpha}$,
\begin{equation*}
\|\widetilde{\mathcal{L}}^{n}g-\langle\mu,g\rangle\cdot1\|\leq c\cdot\theta^{n}
\end{equation*}
for all $n\geq1$. Hence,
\begin{equation}\label{25}
\left|\int f\circ\sigma^{n}\cdot g\mathrm{d}\mu-\int f\mathrm{d}\mu\cdot\int g\mathrm{d}\mu\right|\leq c\cdot\int |f|\mathrm{d}\mu\cdot\theta^{n}
\end{equation}
for any $n\geq1$. For any cylinder $E$ and measurable set $F$, let $f=\chi_{F}$ and $g=\chi_{E}$. Here and below, $\chi$ is the characteristic function.
Then $f\in L^{1}(\Sigma,\mu)$ and $g\in\mathcal{C}^{\alpha}$, it follows from $\eqref{25}$ that for each $n\geq1$,
\begin{equation}
\left|\int\chi_{F}\circ\sigma^{n}\cdot\chi_{E}\mathrm{d}\mu-\int\chi_{F}\mathrm{d}\mu\cdot\int\chi_{E}\mathrm{d}\mu\right|\leq c\cdot\int |\chi_{F}|\mathrm{d}\mu\cdot\theta^{n}.
\end{equation}
That is, for all $n\geq1$,
\begin{equation*}
|\mu(E\cap\sigma^{-n}F)-\mu(E)\mu(F)|\leq c\cdot\mu(F)\theta^{n}
\end{equation*}
\textrm{}
\end{proof}
Let $\mathcal{D}$ be the set of all cylinders. When $\Sigma$ has specification, it follows from Lemma $\ref{24}$ that $(\Sigma,\sigma,\mu)$ is exponentially mixing.
Thus, $\mu$ is summable-mixing with respect to $(\sigma,\mathcal{D})$ (see $\cite[\rm Definition\ 1]{LLVZ2023}$).
By $\cite[\rm Theorem\ 1]{LLVZ2023}$, we immediately obtain the below lemma.
\begin{lemma}\label{26}
Suppose that $\Sigma$ has specification, $h(\Sigma)>0$ and $0^{\infty}\in\Sigma$. Let $I_{n}=I_{\left\lfloor\frac{\log n}{h(\Sigma)}\right\rfloor+1}(0^{\infty})$,
where $\lfloor x\rfloor$ stands for the integer part of $x$. For $\omega\in\Sigma$, denote by $R(N,\omega)$ the number of positive integer $n\leq N$ such that $\sigma^{n}\omega\in I_{n}$.
Then for any given $\epsilon>0$, we have
\begin{equation*}
R(N,\omega)=F(N)+O\left(F^{\frac{1}{2}}(N)(\log F(N))^{\frac{3}{2}+\epsilon}\right)
\end{equation*}
for $\mu$-almost all $\omega\in\Sigma$, where
\begin{equation*}
F(N)=\sum_{n\leq N}\mu(I_{n}).
\end{equation*}
\end{lemma}
\begin{proof}[\rm \textbf{Proof of Theorem $\ref{16}$}]
For simplicity, we denote ``almost all'' by ``a.a.''. Firstly, we prove that
\begin{equation}\label{27}
\limsup_{N\to\infty}\frac{L_{N}(\omega)}{\log N}\leq\frac{1}{h(\Sigma)}\ {\rm for}\ \mu-{\rm a.a.}\ \omega\in\Sigma.
\end{equation}
It suffices to show that
\begin{equation*}
\mu\left(\left\{\omega\in\Sigma:\limsup_{N\to\infty}\frac{L_{N}(\omega)}{\log N}>\frac{1+\epsilon}{h(\Sigma)}\right\}\right)=0
\end{equation*}
for each $\epsilon>0$. Note that
\begin{equation}\label{32}
\begin{aligned}
\left\{\omega\in\Sigma:\limsup_{N\to\infty}\frac{L_{N}(\omega)}{\log N}>\frac{1+\epsilon}{h(\Sigma)}\right\}&\subset\left\{\omega\in\Sigma:\frac{L_{N}(\omega)}{\log N}>\frac{1+\epsilon}{h(\Sigma)}\quad {\rm i.o.}\right\}\\
&\subset\left\{\omega\in\Sigma:l_{n}(\omega)>(1+\epsilon)\frac{\log n}{h(\Sigma)}\quad {\rm i.o.}\right\}\\
&=\bigcap_{N=1}^{\infty}\bigcup_{n=N}^{\infty}\left\{\omega\in\Sigma:l_{n}(\omega)>(1+\epsilon)\frac{\log n}{h(\Sigma)}\right\}\\
&=\bigcap_{N=1}^{\infty}\bigcup_{n=N}^{\infty}\sigma^{-n}I'_{n},
\end{aligned}
\end{equation}
where $I'_{n}=I_{\left\lfloor(1+\epsilon)\frac{\log n}{h(\Sigma)}\right\rfloor+1}(0^{\infty})$. In view of the $\sigma$-invariance of $\mu$, by $\eqref{31}$, we have that
\begin{equation*}
\mu(\sigma^{-n}I'_{n})=\mu(I'_{n})\leq\gamma\cdot e^{-\left(\left\lfloor(1+\epsilon)\frac{\log n}{h(\Sigma)}\right\rfloor+1\right)h(\Sigma)}<\gamma\cdot n^{-(1+\epsilon)}.
\end{equation*}
So $\sum_{n=1}^{\infty}\mu(\sigma^{-n}I'_{n})<+\infty$, applying the first Borel-Cantelli lemma \cite[Lemma 1.2]{GH1998}, we have $\mu\left(\bigcap\limits_{N=1}^{\infty}\bigcup\limits_{n=N}^{\infty}\sigma^{-n}I'_{n}\right)=0$.
By $\eqref{32}$,
\begin{equation*}
\mu\left(\left\{\omega\in\Sigma:\limsup_{N\to\infty}\frac{L_{N}(\omega)}{\log N}>\frac{1+\epsilon}{h(\Sigma)}\right\}\right)=0.
\end{equation*}
Secondly, we show that
\begin{equation}\label{28}
\liminf_{N\to\infty}\frac{L_{N}(\omega)}{\log N}\geq\frac{1}{h(\Sigma)}\ {\rm for}\ \mu-{\rm a.a.}\ \omega\in\Sigma.
\end{equation}
It suffices to prove that
\begin{equation*}
\mu\left(\left\{\omega\in\Sigma:\liminf_{N\to\infty}\frac{L_{N}(\omega)}{\log N}\geq\frac{1-\epsilon}{h(\Sigma)}\right\}\right)=1
\end{equation*}
for any $\epsilon>0$. So we need to verify that for $\mu$-a.a. $\omega\in\Sigma$, when $N$ is large enough, there exists a $n$ such that $1\leq n\leq N$
and
\begin{equation*}
l_{n}(\omega)\geq(1-\epsilon)\frac{\log N}{h(\Sigma)}.
\end{equation*}
Our idea is to prove that for $\mu$-a.a. $\omega\in\Sigma$, when $N$ is large enough, there exists a $n$ such that $N^{1-\epsilon}\leq n\leq N$ and
$\sigma^{n}\omega\in I_{n}$, where $I_{n}=I_{\left\lfloor\frac{\log n}{h(\Sigma)}\right\rfloor+1}(0^{\infty})$. That is, we need to show that
for $\mu$-a.a. $\omega\in\Sigma$, when $N$ is large enough,
\begin{equation*}
\sum_{n=\lfloor N^{1-\epsilon}\rfloor+1}^{N}\chi_{\sigma^{-n}I_{n}}(\omega)\geq1.
\end{equation*}
Note that
\begin{equation*}
\sum_{n=\lfloor N^{1-\epsilon}\rfloor+1}^{N}\chi_{\sigma^{-n}I'_{n}}(\omega)=R(N,\omega)-R(\lfloor N^{1-\epsilon}\rfloor,\omega).
\end{equation*}
It follows from Lemma $\ref{26}$ that
\begin{equation}\label{33}
R(N,\omega)=F(N)+O\left(F^{\frac{1}{2}}(N)(\log F(N))^{2}\right)
\end{equation}
for $\mu$-a.a. $\omega\in\Sigma$, where
\begin{equation*}
F(N)=\sum_{n\leq N}\mu(I_{n}).
\end{equation*}
By Lemma $\ref{21}$ and the $\sigma$-invariance of $\mu$, we have that
$\gamma_{1}^{-1}n^{-1}\leq\mu(I_{n})\leq\gamma_{1}n^{-1}$, where $\gamma_{1}=\gamma\cdot e^{h(\Sigma)}$. Thus,
$\gamma_{1}^{-1}\log N\leq F(N)\leq\gamma_{1}(1+\log N)$. Therefore,
\begin{equation}\label{34}
R(N,\omega)=F(N)+O\left(\log^{\frac{2}{3}}N\right)\ {\rm for}\ \mu-{\rm a.a.}\ \omega\in\Sigma.
\end{equation}
Combining $\eqref{33}$ and $\eqref{34}$, for $\mu$-a.a. $\omega\in\Sigma$, we have
\begin{equation*}
\begin{aligned}
\sum_{n=\lfloor N^{1-\epsilon}\rfloor+1}^{N}\chi_{\sigma^{-n}I_{n}}(\omega)&=F(N)-F(\lfloor N^{1-\epsilon}\rfloor)+O\left(\log^{\frac{2}{3}}N\right)\\
&=\sum_{n=\lfloor N^{1-\epsilon}\rfloor+1}^{N}\mu(I_{n})+O\left(\log^{\frac{2}{3}}N\right)\\
&\geq\gamma_{1}^{-1}\frac{\epsilon}{2}\log N+O\left(\log^{\frac{2}{3}}N\right)\geq1
\end{aligned}
\end{equation*}
for all $N$ large enough, which implies $\eqref{28}$. By $\eqref{27}$ and $\eqref{28}$, the conclusion follows.
\textrm{}
\end{proof}

\section{The Hausdorff dimensions of $\mathcal{U}_{\psi_{0}}(a,b)$ and $\mathcal{U}_{\psi_{0}}^{\ast}(a,b)$}\label{40}
Recall that $\psi_{0}(N)=m^{-N}$, that is, $-\log_{m}\psi_{0}(N)=N$. In this section, we calculate the Hausdorff dimensions of $\mathcal{U}_{\psi_{0}}(a,b)$ and $\mathcal{U}_{\psi_{0}}^{\ast}(a,b)$,
which plays an important role in the proof of Theorem $\ref{113}$. For simplicity, denote $\underline{L}(\omega)=\liminf\limits_{N\to\infty}\frac{L_{N}(\omega)}{N}$ and $\overline{L}(\omega)=\limsup\limits_{N\to\infty}\frac{L_{N}(\omega)}{N}$.
\begin{lemma}\label{55}
Assume that $\Sigma$ has specification, $h(\Sigma)>0$ and $0^{\infty}\in\Sigma$.\\
(\rmnum{1}) If $a<1$ and $b<\frac{a}{1-a}$, then $$\mathcal{U}_{\psi_{0}}(a,b)=\emptyset.$$
(\rmnum{2}) If $a<1$ and $b\geq\frac{a}{1-a}$, then $$\dim_{\rm H}\mathcal{U}^{\ast}_{\psi_{0}}(a,b)=\dim_{\rm H}\mathcal{U}_{\psi_{0}}(a,b)=\frac{b(1-a)-a}{(1+b)(b-a)}\dim_{\rm H}\Sigma.$$
Furthermore, for any positive number $\varrho$, we have that
\begin{equation*}
\dim_{\rm H}(\{\omega\in\Sigma:\underline{L}(\omega)\geq a, b\leq\overline{L}(\omega)\leq b+\varrho\})
\leq\left(\frac{b(1-a)-a}{(1+b)(b-a)}+\frac{\varrho(1-a)^{2}}{a^{3}}\right)\dim_{\rm H}\Sigma.
\end{equation*}
\end{lemma}
\subsection{Proof of Lemma $\ref{55}$ (\rmnum{1})}
Lemma $\ref{55}$ (\rmnum{1}) follows from Propositions $\ref{211}$ and $\ref{51}$.
\begin{proposition}\label{211}
For $\omega=\omega_{1}\cdots\omega_{n}\cdots\in\mathcal{A}^{\mathbb{N}}$, we have $\liminf\limits_{N\to\infty}\frac{L_{N}(\omega)}{N}>1$ if and only if $\omega_{n}=0$ for all $n$ large enough.
\end{proposition}
\begin{proof}[\rm \textbf{Proof}]
The ``if'' part is obviously, we only need to show the ``only if'' part. Since $\liminf\limits_{N\to\infty}\frac{L_{N}(\omega)}{N}>1$, there exists $\epsilon>0$ such that $\liminf\limits_{N\to\infty}\frac{L_{N}(\omega)}{N}>1+\epsilon$. So there exists $N_{0}\in\mathbb{N}$, such that $\frac{L_{N}(\omega)}{N}>1+\epsilon$
for all $N\geq N_{0}$. It means that for every $N\geq N_{0}$, there exists $1\leq n\leq N$, satisfies $l_{n}(\omega)\geq\lfloor(1+\epsilon)N\rfloor+1$. Because of $n+1\leq N+1$ and $n+l_{n}(\omega)\geq n+\lfloor(1+\epsilon)N\rfloor+1\geq N+2$, we obtain that $\omega_{N+1}=0$. Thus, $\omega_{n}=0$ for every $n\geq N_{0}+1$.
\textrm{}
\end{proof}
\begin{proposition}\label{51}
For any $\omega\in\mathcal{A}^{\mathbb{N}}$, we have
\begin{equation}\label{41}
\overline{L}(\omega)=+\infty\ {\rm when}\ \underline{L}(\omega)=1
\end{equation}
and
\begin{equation}\label{42}
\overline{L}(\omega)\geq\frac{\underline{L}(\omega)}{1-\underline{L}(\omega)}\ {\rm when}\ \underline{L}(\omega)<1.
\end{equation}
\end{proposition}
\begin{proof}[\rm \textbf{Proof}]
By Proposition $\ref{211}$, we know that $\underline{L}(\omega)\leq1$ is equivalent to $\omega_{n}\neq0$ for infinitely many $n\in\mathbb{N}$.
Without loss of generality, we assume that $\underline{L}(\omega)>0$. Let $\mathcal{F}=\{n:\omega_{n+1}=0\}$. Since $\overline{L}(\omega)\geq\underline{L}(\omega)>0$, $\mathcal{F}$ is a infinite set. So we can write $\mathcal{F}=\{n'_{k}\}_{k=1}^{\infty}$, where $n'_{1}<n'_{2}<\cdots$. Since $\omega_{n}\neq0$ for infinitely many $n$, for every $k\geq1$, there exists $m'_{k}\geq n'_{k}+1$, satisfies
\begin{equation*}
\omega_{n'_{k}+1}=\cdots=\omega_{m'_{k}}=0,\omega_{m'_{k}+1}\neq0.
\end{equation*}
That is, $l_{n'_{k}}(\omega)=m'_{k}-n'_{k}$. In view of $\overline{L}(\omega)>0$, we have
\begin{equation}\label{43}
\limsup_{k\to\infty}l_{n'_{k}}(\omega)=+\infty.
\end{equation}
Now, we take the maximal subsequences $\{n_{k}\}_{k=1}^{\infty}$ and $\{m_{k}\}_{k=1}^{\infty}$ of $\{n'_{k}\}_{k=1}^{\infty}$ and $\{m'_{k}\}_{k=1}^{\infty}$, respectively, in such a way that the sequence $\{m_{k}-n_{k}\}$ is strictly increasing. More precisely, let $n_{1}=n'_{1}$ and $m_{1}=m'_{1}$. Let $k\geq1$ be such that $n_{k}=n'_{i_{k}}$ and $m_{k}=m'_{i_{k}}$ have been defined. Denote
\begin{equation*}
i_{k+1}=\min\{i>i_{k}:m'_{i}-n'_{i}>m_{k}-n_{k}\}.
\end{equation*}
Then, define
\begin{equation*}
n_{k+1}=n'_{i_{k+1}}\ {\rm and}\ m_{k+1}=m'_{i_{k+1}}.
\end{equation*}
By $\eqref{43}$, the sequence $\{i_{k}\}_{k=1}^{\infty}$ is well defined. Furthermore, by construction, we have
\begin{equation}\label{44}
\overline{L}(\omega)=\limsup_{k\to\infty}\frac{m_{k}-n_{k}}{n_{k}}=\limsup_{k\to\infty}\frac{m_{k}}{n_{k}}-1
\end{equation}
and
\begin{equation}\label{45}
\underline{L}(\omega)=\liminf_{k\to\infty}\frac{m_{k}-n_{k}}{n_{k+1}}\leq\liminf_{k\to\infty}\frac{m_{k}-n_{k}}{m_{k}}=1-\limsup_{k\to\infty}\frac{ n_{k}}{m_{k}}.
\end{equation}
If $\underline{L}(\omega)=1$, it follows from $\eqref{45}$ that $\limsup\limits_{k\to\infty}\frac{n_{k}}{m_{k}}=0$. Thus, $\limsup\limits_{k\to\infty}\frac{m_{k}}{n_{k}}=+\infty$.
It follows from $\eqref{44}$ that $\overline{L}(\omega)=+\infty$. If $\underline{L}(\omega)<1$, we assume that $\overline{L}(\omega)<+\infty$, otherwise, $\eqref{42}$ is trivial.
Then by $\eqref{44}$, we have $\limsup\limits_{k\to\infty}\frac{m_{k}}{n_{k}}<+\infty$. Hence,
\begin{equation}\label{46}
\limsup_{k\to\infty}\frac{m_{k}}{n_{k}}\cdot\limsup_{k\to\infty}\frac{ n_{k}}{m_{k}}\geq1.
\end{equation}
The upshot of $\eqref{44}$, $\eqref{45}$ and $\eqref{46}$ is that $(\overline{L}(\omega)+1)\cdot(1-\underline{L}(\omega))\geq1$.
It follows that $\overline{L}(\omega)\geq\frac{\underline{L}(\omega)}{1-\underline{L}(\omega)}$.
\textrm{}
\end{proof}
\begin{proof}[\rm \textbf{Proof of Lemma $\ref{55}$ (\rmnum{1})}]
For any $\omega\in\mathcal{A}^{\mathbb{N}}$ with $\underline{L}(\omega)\geq a$,
it follows from Propositions $\ref{211}$ and $\ref{51}$ that $\overline{L}(\omega)\geq\frac{a}{1-a}$. Thus, $\mathcal{U}_{\psi_{0}}(a,b)=\emptyset$ when
$a<1$ and $b<\frac{a}{1-a}$.
\textrm{}
\end{proof}
\subsection{Proof of Lemma $\ref{55}$ (\rmnum{2})}
\subsubsection{Upper bound}
In this section, we prove that $$\dim_{\rm H}\mathcal{U}_{\psi_{0}}(a,b)\leq\frac{b(1-a)-a}{(1+b)(b-a)}\dim_{\rm H}\Sigma,\ \forall\ a<1\ {\rm and}\ b\geq\frac{a}{1-a}.$$
\begin{proof}[\rm \textbf{Proof}]
An easy cover argument shows that $\dim_{\rm H}(\{\omega\in\Sigma:\overline{L}(\omega)=+\infty\})=0$.
It follows from Proposition $\ref{51}$ and Proposition $\ref{211}$ that $\mathcal{U}_{\psi_{0}}(1)\subset\{\omega\in\Sigma:\overline{L}(\omega)=+\infty\}$.
Thus, we can assume that $0<b<+\infty$ and we only need to consider the set $\mathcal{U}_{\psi_{0}}(a,b)\setminus\mathcal{U}_{\psi_{0}}(1)$.
For any $\omega\in\mathcal{U}_{\psi_{0}}(a,b)\setminus\mathcal{U}_{\psi_{0}}(1)$, that is, $\omega\in\Sigma$ with $a\leq\underline{L}(\omega)<1$ and $\overline{L}(\omega)=b>0$.
By the proof of Proposition $\ref{51}$, we know that there exists two strictly increasing sequences $\{n_{k}\}_{k=1}^{\infty}$, $\{m_{k}\}_{k=1}^{\infty}$
with $\{m_{k}-n_{k}\}_{k=1}^{\infty}$ is strictly increasing and
\begin{equation*}
\omega_{n_{k}+1}=\cdots=\omega_{m_{k}}=0,\omega_{m_{k}+1}\neq0
\end{equation*}
for each $k\in\mathbb{N}$. Furthermore,
\begin{equation}\label{48}
\limsup_{k\to\infty}\frac{m_{k}-n_{k}}{n_{k}}=\overline{L}(\omega)=b
\end{equation}
and
\begin{equation}\label{49}
\liminf_{k\to\infty}\frac{m_{k}-n_{k}}{n_{k+1}}=\underline{L}(\omega)\geq a.
\end{equation}
Take a subsequence $\{k_{j}\}_{j=1}^{\infty}$ along which the supremum of $\eqref{48}$ is obtained. For simplicity, we still denote $\{n_{k_{j}}\}_{j=1}^{\infty}$, $\{m_{k_{j}}\}_{j=1}^{\infty}$ by $\{n_{k}\}_{k=1}^{\infty}$, $\{m_{k}\}_{k=1}^{\infty}$. We remark that when passing to the subsequence, the first equality in $\eqref{49}$ becomes an inequality. Let $\epsilon\in(0,\frac{a}{2})$, there exists $K_{1}\in\mathbb{N}$, such that for all $k\geq K_{1}$, we have
\begin{equation}\label{410}
(b-\epsilon)n_{k}<m_{k}-n_{k}<(b+\epsilon)n_{k}
\end{equation}
and
\begin{equation}\label{411}
m_{k}-n_{k}>(a-\epsilon)n_{k+1}.
\end{equation}
Combining the second inequality of $\eqref{410}$ and $\eqref{411}$, we obtain that
\begin{equation}\label{412}
m_{k-1}-n_{k-1}>\frac{a-\epsilon}{b+\epsilon}(m_{k}-n_{k}),\ \forall\ k\geq K_{1}+1.
\end{equation}
Since $0<\frac{a-\epsilon}{b+\epsilon}<1$, for all $k$ large enough, we have $(\frac{a-\epsilon}{b+\epsilon})^{k-K_{1}}<\epsilon$. So when $k$ is large enough, the sum of all the lengths of the blocks of 0 in the prefix of length $n_{k}$ in the infinite sequence $\omega_{1}\omega_{2}\cdots$ is at least equal to
\begin{equation}\label{413}
\begin{aligned}
\sum_{j=K_{1}}^{k-1}(m_{j}-n_{j})&\geq\sum_{j=K_{1}}^{k-1}\left(\frac{a-\epsilon}{b+\epsilon}\right)^{k-j}(m_{k}-n_{k})>(b-\epsilon)n_{k}\sum_{j=K_{1}}^{k-1}\left(\frac{a-\epsilon}{b+\epsilon}\right)^{k-j}\\
&>(b-\epsilon)n_{k}\frac{a-\epsilon}{b-a+2\epsilon}(1-\epsilon)
=\left(\frac{ba}{b-a}-\epsilon'\right)n_{k},
\end{aligned}
\end{equation}
where $\epsilon'=\frac{ba}{b-a}-\frac{(b-\epsilon)(a-\epsilon)}{b-a+2\epsilon}+\epsilon\frac{(b-\epsilon)(a-\epsilon)}{b-a+2\epsilon}$.
The first inequality is due to $\eqref{412}$, the second inequality is due to the first inequality of $\eqref{410}$, the third inequality is due to the arithmetic series summation formula and $(\frac{a-\epsilon}{b+\epsilon})^{k-K_{1}}<\epsilon$.
It follows from the first inequality of $\eqref{410}$ that $m_{k}>(b-\epsilon+1)n_{k}\geq(b-\epsilon+1)m_{k-1}$. Thus, $\{m_{k}\}_{k=1}^{\infty}$ increases at least exponentially.
Since $n_{k}\geq m_{k-1}$, the sequence $\{n_{k}\}_{k=1}^{\infty}$ also increases at least exponentially. It follows that there exists $c_{1}>0$, such that
\begin{equation}\label{414}
k\leq c_{1}\log n_{k}, \quad \forall\ k\gg1.
\end{equation}

Now, for any $\delta>0$, let us construct a $\delta$-cover. Let $\Gamma_{n,\epsilon}$ be the set of all $\omega_{1}\cdots\omega_{\lfloor(b-\epsilon+1)n\rfloor}\in\Sigma^{\lfloor(b-\epsilon+1)n\rfloor}$ satisfies $\omega_{n+1}=\cdots=\omega_{\lfloor(b-\epsilon+1)n\rfloor}=0$ and there exists $i$ blocks of 0 of $\omega_{1}\cdots\omega_{n}$, with $i\leq c_{1}\log n$ and the total length of these $i$ blocks is at least equal to $\left(\frac{ba}{b-a}-\epsilon'\right)n$. In view of $\eqref{413}$, $\eqref{414}$ and the first inequality of $\eqref{410}$, for all $N$ large enough,
\begin{equation*}
\bigcup_{n=N}^{\infty}\bigcup_{\omega\in\Gamma_{n,\epsilon}}I_{\lfloor(b-\epsilon+1)n\rfloor}(\omega)
\end{equation*}
is a $\delta$-cover of the set $\mathcal{U}_{\psi_{0}}(a,b)\setminus\mathcal{U}_{\psi_{0}}(1)$. Thus, for any $s\geq0$, we have
\begin{equation}\label{415}
\mathcal{H}_{\delta}^{s}(\mathcal{U}_{\psi_{0}}(a,b)\setminus\mathcal{U}_{\psi_{0}}(1))\leq\sum_{n=N}^{\infty}\sum_{\omega\in\Gamma_{n,\epsilon}}|I_{\lfloor(b-\epsilon+1)n\rfloor}(\omega)|^{s}
\leq\sum_{n=N}^{\infty}m^{-(b-\epsilon+1)ns}\cdot\sharp\Gamma_{n,\epsilon}.
\end{equation}
So, we need to estimate $\sharp\Gamma_{n,\epsilon}$. Note that for any $i$ with $1\leq i\leq c_{1}\log n$ and every $l$ with $\left(\frac{ba}{b-a}-\epsilon'\right)n\leq l\leq n$, there are at most
\begin{equation*}
n(n-1)\cdots(n-i+1)\binom{l-1}{i-1}
\end{equation*}
choices of $i$ blocks of 0 of $\omega_{1}\cdots\omega_{n}$, with the total length of these $i$ blocks is equal to $l$.
Denote by $B=h(\Sigma):=\lim\limits_{n\to\infty}\frac{\log\sharp\Sigma^{n}}{n}$. We claim that the remaining $n-l$ digits has at most
\begin{equation*}
e^{\left(n-\left(\frac{ba}{b-a}-\epsilon'\right)n+iM\right)(B+\epsilon)}
\end{equation*}
choices. In fact, if $n-l<\log n$, then the choices of the remaining $n-l$ digits is at most equal to
\begin{equation*}
m^{n-l}<m^{\log n}<e^{\left(n-\left(\frac{ba}{b-a}-\epsilon'\right)n+iM\right)(B+\epsilon)}.
\end{equation*}
If $n-l\geq\log n$, since the remaining $n-l$ digits could be $i-1$ or $i$ or $i+1$ blocks of $\omega_{1}\cdots\omega_{n}$, we consider three cases.\\
Case (\Rmnum{1}). The remaining $n-l$ digits are $i-1$ blocks of $\omega_{1}\cdots\omega_{n}$, we assume that these $i-1$ blocks is $u_{1},\cdots,u_{i-1}$. Since $\Sigma$ has specification, there exists $v_{1},\cdots,v_{i-2}\in\Sigma^{M}$, such that
$u_{1}v_{1}\cdots u_{i-2}v_{i-2}u_{i-1}\in\Sigma^{n-l+(i-2)M}$. Furthermore, $v_{j}$ ($1\leq j\leq i-2$) is the minimal word (in the sense of lexicographical order) such that $u_{1}v_{1}\cdots u_{j}v_{j}u_{j+1}\in\Sigma^{\ast}$.\\
Case (\Rmnum{2}). The remaining $n-l$ digits are $i$ blocks of $\omega_{1}\cdots\omega_{n}$, we assume that these $i$ blocks is $u_{1},\cdots,u_{i}$. Since $\Sigma$ has specification, there exists $v_{1},\cdots,v_{i-1}\in\Sigma^{M}$, such that
$u_{1}v_{1}\cdots u_{i-1}v_{i-1}u_{i}\in\Sigma^{n-l+(i-1)M}$. Furthermore, $v_{j}$ ($1\leq j\leq i-1$) is the minimal word (in the sense of lexicographical order) such that $u_{1}v_{1}\cdots u_{j}v_{j}u_{j+1}\in\Sigma^{\ast}$.\\
Case (\Rmnum{3}). The remaining $n-l$ digits are $i+1$ blocks of $\omega_{1}\cdots\omega_{n}$, we assume that these $i+1$ blocks is $u_{1},\cdots,u_{i+1}$. Since $\Sigma$ has specification, there exists $v_{1},\cdots,v_{i}\in\Sigma^{M}$, such that
$u_{1}v_{1}\cdots u_{i}v_{i}u_{i+1}\in\Sigma^{n-l+iM}$. Furthermore, $v_{j}$ ($1\leq j\leq i$) is the minimal word (in the sense of lexicographical order) such that $u_{1}v_{1}\cdots u_{j}v_{j}u_{j+1}\in\Sigma^{\ast}$.\\
Thus, for all $n$ large enough, the choices of the remaining $n-l$ digits is at most equal to
\begin{equation}\label{416}
\max\left\{\sharp\Sigma^{n-l+(i-2)M},\sharp\Sigma^{n-l+(i-1)M},\sharp\Sigma^{n-l+iM}\right\}
<e^{(n-l+iM)(B+\epsilon)}
\leq e^{\left(n-\left(\frac{ba}{b-a}-\epsilon'\right)n+iM\right)(B+\epsilon)}.
\end{equation}
The first inequality of $\eqref{416}$ is due to $B=\lim\limits_{n\to\infty}\frac{\log\sharp\Sigma^{n}}{n}$ and $n-l\geq\log n$.
The second inequality of $\eqref{416}$ is due to $l\geq\left(\frac{ba}{b-a}-\epsilon'\right)n$. Therefore, $\sharp\Gamma_{n,\epsilon}$ is not larger than
\begin{equation*}
n^{2c_{1}\log n}c_{1}\log n\cdot e^{\left(n\left(\frac{b-a-ba}{b-a}+\epsilon'\right)+Mc_{1}\log n\right)(B+\epsilon)}.
\end{equation*}
Combining this and $\eqref{415}$, we have to consider the series
\begin{equation}\label{417}
\sum_{n=1}^{\infty}n^{2c_{1}\log n}c_{1}\log n\cdot e^{n(-(b-\epsilon+1)s\log m+\left(\frac{b-a-ba}{b-a}+\epsilon'\right)(B+\epsilon))+Mc_{1}\log n(B+\epsilon)}.
\end{equation}
The critical exponent $s_{0}$ such that $\eqref{417}$ converges if $s>s_{0}$ and diverges if $s<s_{0}$ is given by
\begin{equation*}
s_{0}=(b-\epsilon+1)^{-1}\left(\frac{b(1-a)-a}{b-a}+\epsilon'\right)\frac{B+\epsilon}{\log m}
\end{equation*}
It then follows from $\dim_{\rm H}\Sigma=\frac{h(\Sigma)}{\log m}$ ($\rm{\cite[Proposition\ \Rmnum{3}.1]{HF1967}}$) that
\begin{equation*}
\dim_{\rm H}\mathcal{U}_{\psi_{0}}(a,b)\leq\frac{b(1-a)-a}{(1+b)(b-a)}\dim_{\rm H}\Sigma.
\end{equation*}
Actually, we have proved that, for every $b\geq\frac{a}{1-a}$ and for every $\varrho>0$, we have
\begin{equation}\label{419}
\dim_{\rm H}(\{\omega\in\Sigma: \underline{L}(\omega)\geq a, b\leq\overline{L}(\omega)\leq b+\varrho\})
\leq\left(\frac{b(1-a)-a}{(1+b)(b-a)}+\frac{\varrho(1-a)^{2}}{a^{3}}\right)\dim_{\rm H}\Sigma.
\end{equation}
\textrm{}
\end{proof}
\subsubsection{Lower bound}
In this section, we show that
\begin{equation}\label{426}
\dim_{\rm H}\mathcal{U}_{\psi_{0}}^{\ast}(a,b)\geq\frac{b(1-a)-a}{(1+b)(b-a)}\dim_{\rm H}\Sigma,\ \forall\ a<1\ {\rm and}\ b\geq\frac{a}{1-a}.
\end{equation}
The following proposition enables us to consider only the measure of cylinders, which is helpful in the estimate of the lower bound of Hausdorff dimension.
For any $U\subset\mathcal{A}^{\mathbb{N}}$, let $|U|=\sup\{d(u,v):u,v\in U\}$, which is the diameter of $U$.
\begin{proposition}\label{29} \rm{(MDP)}
Let $E\subset\Sigma$ be a Borel measurable set and $\lambda$ be a Borel measure with $\lambda(E)>0$. If there exists some $s\in[0, \log m]$ such that for any $\omega\in E$,
\begin{equation*}
\liminf_{n\to\infty}\frac{\log\lambda(I_{n}(\omega))}{\log|I_{n}(\omega)|}\geq s,
\end{equation*}
then $\dim_{\rm H}E\geq s$.
\end{proposition}
\begin{proof}[\rm \textbf{Proof}]
The proof is standard, we give the proof here for completeness.
For any $\omega\in E$ and $r\in(0,m^{-1})$, there exists $n\geq1$ such that $m^{-n-1}\leq r<m^{-n}$, hence
$$I_{n}(\omega)=B(\omega,r).$$ Therefore,
\begin{equation*}
\frac{\log\lambda(B(\omega,r))}{\log 2r}=\frac{-\log\lambda(B(\omega,r))}{-\log 2r}=\frac{-\log\lambda(I_{n}(\omega))}{-\log 2r}\geq\frac{-\log\lambda(I_{n}(\omega))}{-\log|I_{n}(\omega)|}=\frac{\log\lambda(I_{n}(\omega))}{\log|I_{n}(\omega)|}.
\end{equation*}
Thus,
\begin{equation*}
\liminf_{r\to0}\frac{\log\lambda(B(\omega,r))}{\log r}\geq\liminf_{n\to\infty}\frac{\log\lambda(I_{n}(\omega))}{\log|I_{n}(\omega)|}\geq s.
\end{equation*}
By \cite[Proposition 2.3]{FK1997}, we have that $\dim_{\rm H}E\geq s$.
\textrm{}
\end{proof}
The following statement shows that if $\Sigma$ has specification and $h(\Sigma)>0$, then the diameter of of $n$th cylinder is comparable to $m^{-n}$.
\begin{proposition}\label{210}
Suppose that $\Sigma$ has specification and $h(\Sigma)>0$. There exists $c^{\ast}>0$, which only depend on $M$, such that for any $\omega\in\Sigma$ and any $n\in\mathbb{N}$,
\begin{equation*}
c^{\ast}m^{-n}\leq|I_{n}(\omega)|\leq m^{-n-1}.
\end{equation*}
\end{proposition}
\begin{proof}[\rm \textbf{Proof}]
By the definition of $I_{n}(\omega)$, we immediately obtain that $|I_{n}(\omega)|\leq m^{-n-1}$. Since $\Sigma$ has specification,
for any $\omega_{1},\omega_{2}\in\Sigma^{\ast}$, there exists $\varepsilon\in\Sigma^{M}$, such that $\omega_{1}\varepsilon\omega_{2}\in\Sigma^{\ast}$.
Because of $$h(\Sigma):=\lim\limits_{n\to\infty}\frac{\log\sharp\Sigma^{n}}{n}>0,$$ there exists $K\in\mathbb{N}$, such that $\sharp\Sigma^{K}>m^{M}\geq\sharp\Sigma^{M}$.
By pigeon-hole principle, there exists $v_{1}\neq v_{2}\in\Sigma^{K}$ and $\varepsilon\in\Sigma^{M}$, satisfies
$\omega_{[1,n]}\varepsilon v_{1},\omega_{[1,n]}\varepsilon v_{2}\in\Sigma^{\ast}$. Note that for any $x\in I_{n+M+K}(\omega_{[1,n]}\varepsilon v_{1})$ and
any $y\in I_{n+M+K}(\omega_{[1,n]}\varepsilon v_{2})$, we have $d(x,y)\geq m^{-n-M-K}$. Thus, $|I_{n}(\omega)|\geq m^{-n-M-K}$.
\textrm{}
\end{proof}
Now, we start proving $\eqref{426}$.
\begin{proof}[\rm \textbf{Proof}]
If $0<a<1$ and $b=\frac{a}{1-a}$, the inequality is trivial. If $0<a<1$ and $b>\frac{a}{1-a}$, we choose an integer $k_{0}$ with $(a^{-1}b)^{k_{0}}>\max\left\{\frac{a(2+b)}{b(1-a)-a},\frac{a(1+b)}{b(b-a)}\right\}$
and let $n_{k}=\left\lfloor (a^{-1}b)^{k+k_{0}}\right\rfloor+1$, $m_{k}=\lfloor(1+b)n_{k}\rfloor+1$. If $a=0$, we choose integers $k_{0}\geq b+1$ and $n_{1}\geq4(\sqrt{2}+1)^{2}$.
Let $n_{k+1}=(k+k_{0}+1)n_{k}$ and $m_{k}=\lfloor(1+b)n_{k}\rfloor+\lfloor\sqrt{n_{k}}\rfloor+1$. It is easy to check that $m_{k}<n_{k+1}$ and $\{m_{k}-n_{k}\}_{k=1}^{\infty}$ is strictly increasing. Furthermore,
$$\lim_{k\to\infty}\frac{m_{k}-n_{k}}{n_{k+1}}=a$$ and $$\lim_{k\to\infty}\frac{m_{k}-n_{k}}{n_{k}}=b.$$

Let $t_{k}$ be the largest integer such that $m_{k}+t_{k}(m_{k}-n_{k})<n_{k+1}$. Since $B>0$, we have $\Sigma^{1}\setminus\{0\}\neq\emptyset$. Fix $\omega\in\Sigma^{1}\setminus\{0\}$, for any $u_{k}^{(1)},\cdots,u_{k}^{(t_{k})}\in\Sigma^{m_{k}-n_{k}-2M-1}$ and any $v_{k}\in\Sigma^{n_{k+1}-m_{k}-t_{k}(m_{k}-n_{k})}$, since $\Sigma$ has specification, there exists $\omega_{k}^{(1)},\cdots,\omega_{k}^{(2t_{k}+3)}\in\Sigma^{M}$, such that
\begin{equation*}
\omega\omega_{k}^{(1)}0^{m_{k}-n_{k}-2M-1}\omega_{k}^{(2)}\omega\omega_{k}^{(3)}u_{k}^{(1)}\omega_{k}^{(4)}\cdots\omega\omega_{k}^{(2t_{k}+1)}u_{k}^{(t_{k})}\omega_{k}^{(2t_{k}+2)}\omega\omega_{k}^{(2t_{k}+3)}v_{k}\in\Sigma^{\ast}.
\end{equation*}
Furthermore, we choose $\omega_{k}^{(1)},\cdots,\omega_{k}^{(2t_{k}+3)}\in\Sigma^{M}$ that satisfies
\begin{enumerate}[(1)]
\item $\omega_{k}^{(1)}$ is the minimal word (in the sense of lexicographical order) such that\\ $\omega\omega_{k}^{(1)}0^{m_{k}-n_{k}-2M-1}\in\Sigma^{\ast}$;
\item For any $1\leq i\leq t_{k}+1$, $\omega_{k}^{(2i)}$ is the minimal word (in the sense of lexicographical order) such that
$\omega\omega_{k}^{(1)}0^{m_{k}-n_{k}-2M-1}\cdots\omega_{k}^{(2i)}\omega\in\Sigma^{\ast}$;
\item For any $1\leq i\leq t_{k}$, $\omega_{k}^{(2i+1)}$ is the minimal word (in the sense of lexicographical order) such that
$\omega\omega_{k}^{(1)}0^{m_{k}-n_{k}-2M-1}\cdots\omega\omega_{k}^{(2i+1)}u_{k}^{(i)}\in\Sigma^{\ast}$;
\item $\omega_{k}^{(2t_{k}+3)}$ is the minimal word (in the sense of lexicographical order) such that \\
$\omega\omega_{k}^{(1)}0^{m_{k}-n_{k}-2M-1}\omega_{k}^{(2)}\cdots\omega\omega_{k}^{(2t_{k}+1)}u_{k}^{(t_{k})}\omega_{k}^{(2t_{k}+2)}\omega\omega_{k}^{(2t_{k}+3)}v_{k}\in\Sigma^{\ast}$.
\end{enumerate}
According to the above rules, for every $u_{k}^{(1)},\cdots,u_{k}^{(t_{k})}\in\Sigma^{m_{k}-n_{k}-2M-1}$ and every $v_{k}\in\Sigma^{n_{k+1}-m_{k}-t_{k}(m_{k}-n_{k})}$, the choice of $\omega_{k}^{(1)},\cdots,\omega_{k}^{(2t_{k}+3)}\in\Sigma^{M}$ is determined, we put
\begin{equation*}
\begin{aligned}
\mathcal{W}_{k}=\left\{\omega\omega_{k}^{(1)}0^{m_{k}-n_{k}-2M-1}\omega_{k}^{(2)}\cdots\omega\omega_{k}^{(2t_{k}+1)}u_{k}^{(t_{k})}\omega_{k}^{(2t_{k}+2)}\omega\omega_{k}^{(2t_{k}+3)}v_{k}\in\Sigma^{\ast}:\right.\\
\left. v_{k}\in\Sigma^{n_{k+1}-m_{k}-t_{k}(m_{k}-n_{k})},u_{k}^{(i)}\in\Sigma^{m_{k}-n_{k}-2M-1},\ \forall 1\leq i\leq t_{k}\right\}.
\end{aligned}
\end{equation*}
Note that every element of $\mathcal{W}_{k}$ is of length $n_{k+1}-n_{k}+M+1$.
Again, since $\Sigma$ has specification, for any $k\geq1$ and any $u_{i}\in\mathcal{W}_{i}$ $(1\leq i\leq k)$, there exists $\varepsilon_{1},\cdots,\varepsilon_{k}\in\Sigma^{M}$, such that $0^{n_{1}}\varepsilon_{1}u_{1}\varepsilon_{2}u_{2}\cdots\varepsilon_{k}u_{k}\in\Sigma^{\ast}$.
Here we choose $\varepsilon_{1},\cdots,\varepsilon_{k}\in\Sigma^{M}$, which satisfies for any $1\leq i\leq k$, $\varepsilon_{i}$ is the minimal word (in the sense of lexicographic order) such that $0^{n_{1}}\varepsilon_{1}u_{1}\cdots\varepsilon_{i}u_{i}\in\Sigma^{\ast}$. Let
\begin{equation*}
G_{k}=\{0^{n_{1}}\varepsilon_{1}u_{1}\cdots\varepsilon_{k}u_{k}\in\Sigma^{\ast}:u_{i}\in\mathcal{W}_{i},\forall\ 1\leq i\leq k\}.
\end{equation*}
Note that every element of $G_{k}$ is of length $n_{k+1}+(2M+1)k$. Define
\begin{equation*}
E_{0}=I_{n_{1}}(0^{n_{1}})
\end{equation*}
and
\begin{equation*}
E_{k}=\bigcup_{w\in G_{k}}I_{n_{k+1}+(2M+1)k}(w)
\end{equation*}
for each $k\geq1$. Put
\begin{equation*}
E=\bigcap_{k=0}^{\infty}E_{k},
\end{equation*}
we show that $E\subset\mathcal{U}_{\psi_{0}}^{\ast}(a,b)$ and $\dim_{\rm H}E\geq\frac{b(1-a)-a}{(1+b)(b-a)}\dim_{\rm H}\Sigma$. Denote $N_{k}:=n_{k+1}+(2M+1)k$.
For any $\omega\in E$, we prove that $\underline{L}(\omega)=a$ and $\overline{L}(\omega)=b$. Since the element of $E$ are of the form $0^{n_{1}}\varepsilon_{1}u_{1}\cdots\varepsilon_{k}u_{k}\cdots$ with $u_{k}\in\mathcal{W}_{k}$ $(k\geq1)$ and $\varepsilon_{k}\in\Sigma^{M}$ is the minimal word such that $0^{n_{1}}\varepsilon_{1}u_{1}\cdots\varepsilon_{k}u_{k}\in\Sigma^{\ast}$, by the definition of $\mathcal{W}_{k}$, we have the following two facts:\\
(\rmnum{1}) If $N_{k}\leq N\leq N_{k}+2M$, then
\begin{equation*}
m_{k}-n_{k}-2M-1\leq L_{N}(\omega)\leq m_{k}-n_{k}+2M.
\end{equation*}
(\rmnum{2}) If $N_{k}+2M<N<N_{k+1}$, then
\begin{equation*}
m_{k+1}-n_{k+1}-2M-1\leq L_{N}(\omega)\leq m_{k+1}-n_{k+1}+2M.
\end{equation*}
Thus, we have
\begin{equation*}
\underline{L}(\omega):=\liminf_{N\to\infty}\frac{L_{N}(\omega)}{N}\geq\liminf_{k\to\infty}\frac{m_{k}-n_{k}}{N_{k}}=\lim_{k\to\infty}\frac{m_{k}-n_{k}}{n_{k+1}}=a
\end{equation*}
and
\begin{equation*}
\overline{L}(\omega):=\limsup_{N\to\infty}\frac{L_{N}(\omega)}{N}\leq\limsup_{k\to\infty}\frac{m_{k+1}-n_{k+1}}{N_{k}}=\lim_{k\to\infty}\frac{m_{k+1}-n_{k+1}}{n_{k+1}}=b. \end{equation*}
Note that
\begin{equation*}
\lim_{k\to\infty}\frac{L_{N_{k}}(\omega)}{N_{k}}=\lim_{k\to\infty}\frac{m_{k}-n_{k}}{n_{k+1}}=a
\end{equation*}
and
\begin{equation*}
\lim_{k\to\infty}\frac{L_{N_{k}+2M+1}(\omega)}{N_{k}+2M+1}=\lim_{k\to\infty}\frac{m_{k+1}-n_{k+1}}{n_{k+1}}=b.
\end{equation*}
Therefore, $\underline{L}(\omega)=a$ and $\overline{L}(\omega)=b$. It remain to prove that $\dim_{\rm H}E\geq\frac{b(1-a)-a}{(1+b)(b-a)}\dim_{\rm H}\Sigma$. Firstly, we assign a mass distribution on $E$. We first assign a mass distribution on all cylinders. Let $\lambda(I_{N}(0^{N}))=1$ for any $1\leq N\leq n_{1}$. For any cylinder $I_{N_{1}}\in E_{1}$, we define
\begin{equation*}
\lambda(I_{N_{1}})=\frac{1}{\sharp\mathcal{W}_{1}}\lambda(I_{n_{1}}(0^{n_{1}}))=\frac{1}{\sharp\mathcal{W}_{1}}.
\end{equation*}
Suppose that the measure on $I_{N_{k}}\in E_{k}$ $(k\geq1)$ is well defined, then the measure on $I_{N_{k+1}}\in E_{k+1}$ is defined as
\begin{equation*}
\lambda(I_{N_{k+1}})=\frac{1}{\sharp\mathcal{W}_{k+1}}\lambda(I_{N_{k}})=\prod_{i=1}^{k+1}\frac{1}{\sharp\mathcal{W}_{i}}.
\end{equation*}
For any $n_{1}<N<N_{1}$, define
\begin{equation*}
\lambda(I_{N})=\sum_{I_{N_{1}}:I_{N_{1}}\cap I_{N}\neq\emptyset}\lambda(I_{N_{1}}).
\end{equation*}
For any $N_{k}<N<N_{k+1}$, define
\begin{equation*}
\lambda(I_{N})=\sum_{I_{N_{k+1}}:I_{N_{k+1}}\cap I_{N}\neq\emptyset}\lambda(I_{N_{k+1}}).
\end{equation*}
Note that
\begin{equation*}
\lambda(I_{N})=\sum_{I_{N+1}:I_{N+1}\cap I_{N}\neq\emptyset}\lambda(I_{N+1}),
\end{equation*}
thus, $\lambda$ can be extended to a measure on $E$. Secondly, in view of Proposition $\ref{29}$, we only need to show that for any $\omega\in E$,
\begin{equation*}
\liminf_{N\to\infty}\frac{\log\lambda(I_{N}(\omega))}{\log|I_{N}(\omega)|}\geq\frac{b(1-a)-a}{(1+b)(b-a)}\dim_{\rm H}\Sigma,
\end{equation*}
where $I_{N}(\omega)$ is the $N$-th cylinder contain $\omega$. It follows from Proposition $\ref{210}$ that
\begin{equation}\label{420}
\liminf_{N\to\infty}\frac{\log\lambda(I_{N}(\omega))}{\log|I_{N}(\omega)|}=\liminf_{N\to\infty}\frac{-\log\lambda(I_{N}(\omega))}{N\log m}.
\end{equation}
If $N_{k}\leq N\leq M_{k}:=m_{k+1}+(2M+1)(k+1)$, then $\lambda(I_{N}(\omega))=\lambda(I_{N_{k}}(\omega))=\prod_{i=1}^{k}\frac{1}{\sharp\mathcal{W}_{i}}$. Thus,
\begin{equation}\label{421}
\begin{aligned}
\frac{-\log\lambda(I_{N}(\omega))}{N}\geq\frac{\sum_{i=1}^{k}\log\sharp\mathcal{W}_{i}}{M_{k}}.
\end{aligned}
\end{equation}
Note that
\begin{equation}\label{422}
\liminf_{k\to\infty}\frac{\sum_{i=1}^{k}\log\sharp\mathcal{W}_{i}}{M_{k}}=\liminf_{k\to\infty}\frac{\sum_{i=1}^{k}\log\sharp\mathcal{W}_{i}}{m_{k+1}}
\geq\liminf_{k\to\infty}\frac{\log\sharp\mathcal{W}_{k}}{m_{k+1}-m_{k}}.
\end{equation}
The equality follows from $\lim\limits_{k\to\infty}\frac{k}{m_{k}}=0$. The inequality follows from $\lim\limits_{k\to\infty}m_{k}=+\infty$ and the definition of limit inferior.
By the definition of $\mathcal{W}_{k}$, we have that
\begin{equation*}
\sharp\mathcal{W}_{k}=\left(\sharp\Sigma^{m_{k}-n_{k}-2M-1}\right)^{t_{k}}\cdot\sharp\Sigma^{n_{k+1}-m_{k}-t_{k}(m_{k}-n_{k})}.
\end{equation*}
Since $B=\lim\limits_{n\to\infty}\frac{\log\sharp\Sigma^{n}}{n}$, we obtain that $\lim\limits_{k\to\infty}\frac{\log\sharp\mathcal{W}_{k}}{n_{k+1}-m_{k}}=B$. Hence,
\begin{equation}\label{423}
\lim_{k\to\infty}\frac{\log\sharp\mathcal{W}_{k}}{m_{k+1}-m_{k}}=B\cdot\lim_{k\to\infty}\frac{n_{k+1}-m_{k}}{m_{k+1}-m_{k}}=\frac{b(1-a)-a}{(1+b)(b-a)}h(\Sigma). \end{equation}
The last equality follows from our choice of sequences $\{n_{k}\}$, $\{m_{k}\}$. For any $\delta>0$, the combination of $\eqref{421}$, $\eqref{422}$ and $\eqref{423}$ gives
\begin{equation}\label{424}
\frac{-\log\lambda(I_{N}(\omega))}{N\log m}>(\log m)^{-1}\left(\frac{b(1-a)-a}{(1+b)(b-a)}h(\Sigma)-\delta\right)
\end{equation}
for all $N$ large enough with $N_{k}\leq N\leq M_{k}$. \\ Recall that $N_{k+1}=n_{k+2}+(2M+1)(k+1)$ and $M_{k}=m_{k+1}+(2M+1)(k+1)$.
Thus, $$N_{k+1}=M_{k}+n_{k+2}-m_{k+1}\leq M_{k}+(t_{k+1}+1)(m_{k+1}-n_{k+1}).$$ If $M_{k}<N<N_{k+1}$, then $M_{k}<N<M_{k}+(t_{k+1}+1)(m_{k+1}-n_{k+1})$.
So there exists $0\leq q\leq t_{k+1}$ and $0\leq r<m_{k+1}-n_{k+1}$, such that $N=M_{k}+q(m_{k+1}-n_{k+1})+r$. \\
\textbf{Case 1}. If $r\leq M$, then $N\leq M_{k}+q(m_{k+1}-n_{k+1})+M$. So
\begin{equation*}
\begin{aligned}
\lambda(I_{N}(\omega))&\leq\lambda(I_{M_{k}+q(m_{k+1}-n_{k+1})}(\omega))\\
&=\left(\sharp\Sigma^{m_{k+1}-n_{k+1}-2M-1}\right)^{-q}\lambda(I_{M_{k}}(\omega)).
\end{aligned}
\end{equation*}
Thus,
\begin{equation*}
\frac{-\log\lambda(I_{N}(\omega))}{N}\geq\frac{-\log\lambda(I_{ M_{k}}(\omega))+q\log\sharp\Sigma^{m_{k}-n_{k}-2M-1}}{M_{k}+q(m_{k+1}-n_{k+1})+M}.
\end{equation*}
In view of $\eqref{423}$, $\lim\limits_{k\to\infty}m_{k}=+\infty$ and the definition of limit superior, we have
\begin{equation*}
\begin{aligned}
\limsup_{k\to\infty}\frac{-\log\lambda(I_{M_{k}}(\omega))}{M_{k}}&=\limsup_{k\to\infty}\frac{\sum_{i=1}^{k}\log\sharp\mathcal{W}_{i}}{m_{k+1}}
\leq\limsup_{k\to\infty}\frac{\log\sharp\mathcal{W}_{k}}{m_{k+1}-m_{k}}\\&=\lim_{k\to\infty}\frac{\log\sharp\mathcal{W}_{k}}{m_{k+1}-m_{k}}=\frac{b(1-a)-a}{(1+b)(b-a)}B
\leq B.
\end{aligned}
\end{equation*}
Thus, for every sufficiently small $\delta>0$, for all $k$ large enough, we have
\begin{equation*}
(B+\delta)^{-1}(-\log\lambda(I_{M_{k}}(\omega)))<M_{k}
\end{equation*}
This together with $B=\lim\limits_{n\to\infty}\frac{\log\sharp\Sigma^{n}}{n}$ gives
\begin{equation*}
\begin{aligned}
\frac{-\log\lambda(I_{N}(\omega))}{N}&\geq\frac{-\log\lambda(I_{M_{k}}(\omega))+(B-\delta)q(m_{k+1}-n_{k+1})}{ M_{k}+ q(m_{k+1}-n_{k+1})+M}\\&\geq\frac{B-\delta}{1+\delta}\cdot\frac{(B+\delta)^{-1}(-\log\lambda(I_{M_{k}}(\omega)))+q(m_{k+1}-n_{k+1})}{ M_{k}+q(m_{k+1}-n_{k+1})}\\ &\geq\frac{B-\delta}{(B+\delta)(1+\delta)}\cdot\frac{-\log\lambda(I_{M_{k}}(\omega))}{M_{k}}\\
&>\frac{B-\delta}{(B+\delta)(1+\delta)}\left(\frac{b(1-a)-a}{(1+b)(b-a)}h(\Sigma)-\delta\right)
\end{aligned}
\end{equation*}
for all $N$ large enough. The third inequality follows from the inequality
\begin{equation*}
\frac{y+s}{z+s}\geq\frac{y}{z},0\leq y<z,s\geq0.
\end{equation*}
The last inequality follows from $\eqref{424}$.\\
\textbf{Case 2}. If $0\leq q\leq t_{k+1}-1$ and $r\geq m_{k+1}-n_{k+1}-2M-1$, we have
\begin{equation*}
\lambda(I_{N}(\omega))=\left(\sharp\Sigma^{m_{k+1}-n_{k+1}-2M-1}\right)^{-(q+1)}\lambda(I_{M_{k}}(\omega)).
\end{equation*}
Thus,
\begin{equation*}
\frac{-\log\lambda(I_{N}(\omega))}{N}>\frac{-\log\lambda(I_{ M_{k}}(\omega))+(q+1)\log\sharp\Sigma^{m_{k+1}-n_{k+1}-2M-1}}{M_{k}+(q+1)(m_{k+1}-n_{k+1})}.
\end{equation*}
Similar with \textbf{Case 1}, for every sufficiently small positive number $\delta$, for $N$ large enough, we have
\begin{equation*}
\frac{-\log\lambda(I_{N}(\omega))}{N}>\frac{B-\delta}{B+\delta}\cdot\left(\frac{b(1-a)-a}{(1+b)(b-a)}h(\Sigma)-\delta\right).
\end{equation*}
\textbf{Case 3}. If $0\leq q\leq t_{k+1}-1$ and $M<r<m_{k+1}-n_{k+1}-2M-1$, then
\begin{equation*}
\begin{aligned}
\lambda(I_{N}(\omega))
&\leq\lambda(I_{M_{k}}(\omega))\cdot(\sharp\Sigma^{m_{k+1}-n_{k+1}-2M-1})^{-q-1}\cdot\sharp\Sigma^{m_{k+1}-n_{k+1}-2M-1-r},\\
&=\lambda(I_{M_{k}}(\omega))\cdot(\sharp\Sigma^{m_{k+1}-n_{k+1}-2M-1})^{-q}\cdot\frac{\sharp\Sigma^{m_{k+1}-n_{k+1}-2M-1-r}}{\Sigma^{m_{k+1}-n_{k+1}-2M-1}}.
\end{aligned}
\end{equation*}
Since $\Sigma$ has specification, we have
\begin{equation}\label{425}
\sharp\Sigma^{i}\geq\sharp\Sigma^{i-j}\cdot\sharp\Sigma^{j-M}
\end{equation}
for all $i>j>M$. Hence,
\begin{equation*}
\lambda(I_{N}(\omega))\leq\lambda(I_{M_{k}}(\omega))\cdot(\sharp\Sigma^{m_{k+1}-n_{k+1}-2M-1})^{-q}\cdot(\sharp\Sigma^{r-M})^{-1}.
\end{equation*}
Thus, for every sufficiently small $\delta>0$, for all $N$ large enough, we have
\begin{equation*}
\begin{aligned}
\frac{\lambda(I_{N}(\omega))}{N}
&\geq\frac{-\log\lambda(I_{M_{k}}(\omega))+q\cdot\log\sharp\Sigma^{m_{k+1}-n_{k+1}-2M-1}+\log\sharp\Sigma^{r-M}}{M_{k}+q(m_{k+1}-n_{k+1})+r}\\
&\geq\frac{-\log\lambda(I_{M_{k}}(\omega))+q\cdot(1-\delta)\log\sharp\Sigma^{m_{k+1}-n_{k+1}}+\log\sharp\Sigma^{r-M}}{M_{k}+q(m_{k+1}-n_{k+1})+r}\\
&\geq\frac{-\log\lambda(I_{M_{k}}(\omega))+(1-\delta)\log\sharp\Sigma^{q(m_{k+1}-n_{k+1})+r-M}}{M_{k}+q(m_{k+1}-n_{k+1})+r}\\
&\geq\frac{1}{1+\delta}\cdot\frac{-\log\lambda(I_{M_{k}}(\omega))+(1-\delta)\log\sharp\Sigma^{q(m_{k+1}-n_{k+1})+r-M}}{M_{k}+q(m_{k+1}-n_{k+1})+r-M}.
\end{aligned}
\end{equation*}
The second inequality follows from $\lim\limits_{n\to\infty}\frac{\log\sharp\Sigma^{n}}{n}$ exists and $\lim\limits_{k\to\infty}(m_{k+1}-n_{k+1})=+\infty$.
The third inequality follows from $\eqref{11}$ and $\sharp\Sigma^{r-M}\geq1$. Similar with \textbf{Case 1}, for all $N$ large enough, we have
\begin{equation*}
\frac{\lambda(I_{N}(\omega))}{N}>\frac{(1-\delta)(B-\delta)}{(1+\delta)(B+\delta)}\left(\frac{b(1-a)-a}{(1+b)(b-a)}h(\Sigma)-\delta\right).
\end{equation*}
\textbf{Case 4}. If $q=t_{k+1}$ and $M<r<n_{k+2}-m_{k+1}-t_{k+1}(m_{k+1}-n_{k+1})$, then
\begin{equation*}
\begin{aligned}
\lambda(I_{N}(\omega))&\leq\sharp\Sigma^{n_{k+2}-m_{k+1}-t_{k+1}(m_{k+1}-n_{k+1})-r}\cdot\lambda(I_{N_{k+1}}(\omega))\\
&=\lambda(I_{M_{k}}(\omega))\cdot(\sharp\Sigma^{m_{k+1}-n_{k+1}-2M-1})^{-t_{k+1}}\cdot\frac{\sharp\Sigma^{n_{k+2}-m_{k+1}-t_{k+1}(m_{k+1}-n_{k+1})-r}}{\sharp\Sigma^{n_{k+2}-m_{k+1}-t_{k+1}(m_{k+1}-n_{k+1})}}\\
&\leq\lambda(I_{M_{k}}(\omega))\cdot(\sharp\Sigma^{m_{k+1}-n_{k+1}-2M-1})^{-t_{k+1}}\cdot(\sharp\Sigma^{r-M})^{-1}.
\end{aligned}
\end{equation*}
The last inequality follows from $\eqref{425}$. Similar with \textbf{Case 3}, when $N$ is large enough, we have
\begin{equation*}
\frac{\lambda(I_{N}(\omega))}{N}>\frac{(1-\delta)(B-\delta)}{(1+\delta)(B+\delta)}\left(\frac{b(1-a)-a}{(1+b)(b-a)}h(\Sigma)-\delta\right).
\end{equation*}
For every sufficient small $\delta$, the combination of $\eqref{424}$ and \textbf{Cases 1-4} gives
\begin{equation*}
\frac{\lambda(I_{N}(\omega))}{N\log m}>(\log m)^{-1}\frac{(1-\delta)(B-\delta)}{(1+\delta)(B+\delta)}\left(\frac{b(1-a)-a}{(1+b)(b-a)}h(\Sigma)-\delta\right)
\end{equation*}
for all $N$ large enough. Thus, we have
\begin{equation}\label{427}
\liminf_{N\to\infty}\frac{\lambda(I_{N}(\omega))}{N\log m}\geq\frac{b(1-a)-a}{(1+b)(b-a)}\cdot\frac{h(\Sigma)}{\log m}=\frac{b(1-a)-a}{(1+b)(b-a)}\dim_{\rm H}\Sigma.
\end{equation}
The equality follows from $\dim_{\rm H}\Sigma=\frac{h(\Sigma)}{\log m}$ (see $\rm{\cite[Proposition\ \Rmnum{3}.1]{HF1967}}$).
The combination of $\eqref{420}$, $\eqref{427}$ and Proposition $\ref{29}$ gives $\eqref{426}$.
\textrm{}
\end{proof}
\section{Proof of Theorem $\ref{113}$ and Corollary $\ref{17}$}\label{50}
In this section, we give the proof of Theorem $\ref{113}$ and Corollary $\ref{17}$.
\subsection{Proof of Theorem $\ref{113}$ Case \rmnum{3}}
\begin{proof}[\rm \textbf{Proof}]
We consider six cases.\\ \textbf{Case 1}: $a=b=\infty$. Since $\lim\limits_{N\to\infty}\frac{-\log_{m}\psi(N)}{N}=0$, we have $\overline{L}(\omega)>0$ ($\underline{L}(\omega)>0$) implies
$\liminf\limits_{N\to\infty}\frac{L_{N}(\omega)}{-\log_{m}\psi(N)}=+\infty$ ($\limsup\limits_{N\to\infty}\frac{L_{N}(\omega)}{-\log_{m}\psi(N)}=+\infty$) respectively.
Thus, for any $0<a'<1$ and $b'\geq\frac{a'}{1-a'}$, we have $\mathcal{U}_{\psi_{0}}^{\ast}(a',b')\subset\mathcal{U}_{\psi}^{\ast}(\infty,\infty)$. It follows from Lemma $\ref{55}$ that
\begin{equation*}
\dim_{\rm H}(\mathcal{U}_{\psi}^{\ast}(\infty,\infty))\geq\frac{b'-b'a'-a'}{(1+b')(b'-a')}\dim_{\rm H}\Sigma,\ \forall\ 0<a'<1\ {\rm and}\ b'\geq\frac{a'}{1-a'}.
\end{equation*}
Therefore, $\dim_{\rm H}(\mathcal{U}_{\psi}^{\ast}(\infty,\infty))\geq\dim_{\rm H}\Sigma$. Thus, $\dim_{\rm H}(\mathcal{U}_{\psi}^{\ast}(\infty,\infty))=\dim_{\rm H}\Sigma$.\\
\textbf{Case 2}: $0<a\leq b<\infty$. Firstly, we construct a Cantor-type subset of $\mathcal{U}_{\psi}^{\ast}(a,b)$. Denote by $\Phi(x)=-\log_{m}\psi(x)$.
Fix a integer $P\geq3$. We define the following two sequences $\{n_{k}\}_{k\geq0}$ and $\{d_{k}\}_{k\geq0}$ of integers as follows: choose $n_{0}$ such that
\begin{equation*}
\lfloor b \Phi(n_{0})\rfloor\geq P+2+3M
\end{equation*}
and then take $d_{0}=\lfloor b\Phi(n_{0})\rfloor$, for any $k\geq1$, $\{n_{k}\}$ and $\{d_{k}\}$ are recursively defined as
\begin{equation*}
n_{k}:=\left\lfloor\Phi^{-1}\left(\frac{b}{a}\Phi(n_{k-1})\right)\right\rfloor+P d_{k-1}
\end{equation*}
and
\begin{equation*}
d_{k}:=\lfloor b\Phi(n_{k})\rfloor.
\end{equation*}
Here, $\{n_{k}\}$ is well defined because of the strictly increasing property of the function $\Phi$. Note that
\begin{equation}\label{61}
\lim_{k\to\infty}\frac{\Phi(n_{k})}{\Phi(n_{k-1})}=\frac{b}{a}.
\end{equation}
In fact, since $a\leq b$ and $\Phi$ is strictly increasing, we have $\Phi^{-1}\left(\frac{b}{a}\Phi(n_{k-1})\right)\geq n_{k-1}$. Thus, $\left\lfloor\Phi^{-1}\left(\frac{b}{a}\Phi(n_{k-1})\right)\right\rfloor\geq n_{k-1}$. Therefore,
\begin{equation}\label{62}
n_{k}\leq\left\lfloor\Phi^{-1}\left(\frac{b}{a}\Phi(n_{k-1})\right)\right\rfloor+Pb\cdot\Phi\left(\left\lfloor\Phi^{-1}\left(\frac{b}{a}\Phi(n_{k-1})\right)\right\rfloor\right). \end{equation}
Combining $\eqref{62}$ and $\lim\limits_{N\to\infty}\frac{\Phi(N+\Phi(N))}{\Phi(N)}=1$, we obtain that
\begin{equation*}
\lim_{k\to\infty}\frac{\Phi(n_{k})}{\Phi\left(\left\lfloor\Phi^{-1}\left(\frac{b}{a}\Phi(n_{k-1})\right)\right\rfloor\right)}=1.
\end{equation*}
Thus,
\begin{equation*}
\lim_{k\to\infty}\frac{\Phi(n_{k})}{\Phi(n_{k-1})}=\frac{b}{a}.
\end{equation*}
In the following, we construct a Cantor-type subset of $\mathcal{U}_{\psi}^{\ast}(a,b)$. Denote by $B=h(\Sigma):=\lim\limits_{n\to\infty}\frac{\log\sharp\Sigma^{n}}{n}$. Since $B>0$, we have $\Sigma^{1}\setminus\{0\}\neq\emptyset$.
Fix $\omega\in\Sigma^{1}\setminus\{0\}$. For each $k\in\mathbb{N}$, we write
\begin{equation*}
n_{k}-n_{k-1}=d_{k-1}(l_{k}+1)+r_{k},
\end{equation*}
where
\begin{equation*}
l_{k}=\left\lfloor\frac{n_{k}-n_{k-1}}{d_{k-1}}\right\rfloor-1, \quad 0\leq r_{k}<d_{k-1}.
\end{equation*}
Let
\begin{equation*}
v_{k}=\begin{cases} 0^{r_{k}}, &{\rm if}\ 0\leq r_{k}\leq P, \\ \omega0^{r_{k}-1}, &{\rm if}\ P<r_{k}<d_{k-1}.
\end{cases}
\end{equation*}
Since $\Sigma$ has specification, for any $u_{k}^{(1)},u_{k}^{(2)},\cdots,u_{k}^{(l_{k})}\in\Sigma^{d_{k-1}-1-2M}$, there exists $\varepsilon_{k}^{(1)},\varepsilon_{k}^{(2)},$\\ $\cdots,\varepsilon_{k}^{(2l_{k}+2)}\in\Sigma^{M}$, such that
\begin{equation*}
\omega\varepsilon_{k}^{(1)}0^{d_{k-1}-1-3M}\varepsilon_{k}^{(2)}\omega\varepsilon_{k}^{(3)}u_{k}^{(1)}\varepsilon_{k}^{(4)}\omega\varepsilon_{k}^{(5)}u_{k}^{(2)}\varepsilon_{k}^{(6)}\cdots\omega\varepsilon_{k}^{(2l_{k}+1)}u_{k}^{(l_{k})}\varepsilon_{k}^{(2l_{k}+2)}v_{k}\in\Sigma^{\ast}.
\end{equation*}
Similar with the proof of Lemma $\ref{55}$, we choose $\varepsilon_{k}^{(1)},\cdots,\varepsilon_{k}^{(2l_{k}+2)}\in\Sigma^{M}$ that satisfies
\begin{enumerate}[(i)]
\item For any $0\leq i\leq l_{k}$, $\varepsilon_{k}^{(2i+1)}$ is the minimal word (in the sense of lexicographical order) such that $\omega \varepsilon_{k}^{(1)}0^{d_{k-1}-1-3M}\varepsilon_{k}^{(2)}\omega\varepsilon_{k}^{(3)}\cdots\varepsilon_{k}^{(2i)}\omega\varepsilon_{k}^{(2i+1)}u_{k}^{(i)}\in\Sigma^{\ast}$;
\item For any $1\leq i\leq l_{k}$, $\varepsilon_{k}^{(2i)}$ is the minimal word (in the sense of lexicographical order) such that $\omega \varepsilon_{k}^{(1)}0^{d_{k-1}-1-3M}\varepsilon_{k}^{(2)}\omega\varepsilon_{k}^{(3)}u_{k}^{(1)}\varepsilon_{k}^{(4)}\cdots\varepsilon_{k}^{(2i-1)}u_{k}^{(i-1)}\varepsilon_{k}^{(2i )}\omega\in\Sigma^{\ast}$;
\item $\varepsilon_{k}^{(2l_{k}+2)}$ is the minimal word (in the sense of lexicographical order) such that \\ $\omega\varepsilon_{k}^{(1)}0^{d_{k-1}-1-3M}\varepsilon_{k}^{(2)}\omega\varepsilon_{k}^{(3)}u_{k}^{(1)}\varepsilon_{k}^{(4)}\omega\varepsilon_{k}^{(5)}u_{k}^{(2)}\varepsilon_{k}^{(6)}\cdots\omega\varepsilon_{k}^{(2l_{k}+1)}u_{k}^{(l_{k})}\varepsilon_{k}^{(2l_{k}+2)}v_{k}\in\Sigma^{\ast}$.
\end{enumerate}
According to the above rules, for every $u_{k}^{(1)},\cdots,u_{k}^{(l_{k})}\in\Sigma^{d_{k-1}-1-2M}$, the choice of $\varepsilon_{k}^{(1)},\cdots,\varepsilon_{k}^{(2l_{k}+2)}\in\Sigma^{M}$ is determined. Put
\begin{equation*}
\begin{aligned}
\mathcal{W}_{k}=\left\{\omega\varepsilon_{k}^{(1)}0^{d_{k-1}-1-3M}\varepsilon_{k}^{(2)}\omega\varepsilon_{k}^{(3)}u_{k}^{(1)}\varepsilon_{k}^{(4)}\omega\varepsilon_{k}^{(5)}u_{k}^{(2)}\varepsilon_{k}^{(6)}\cdots\omega\varepsilon_{k}^{(2l_{k}+1)}u_{k}^{(l_{k})}\varepsilon_{k}^{(2l_{k}+2)}v_{k}\in\Sigma^{\ast}:  \right. \\ \left. u_{k}^{(i)}\in\Sigma^{d_{k-1}-1-2M},\forall\ 1\leq i\leq l_{k} \right\}.
\end{aligned}
\end{equation*}
Note that every element of $\mathcal{W}_{k}$ is of length $n_{k}-n_{k-1}-M$. Furthermore, since $\Sigma$ has specification,
for any $k\geq1$ and any $u_{i}\in\mathcal{W}_{i}(1\leq i\leq k)$, there exists $\varepsilon_{1},\cdots,\varepsilon_{k}\in\Sigma^{M}$, such that
$0^{n_{0}}\varepsilon_{1}u_{1}\varepsilon_{2}u_{2}\cdots\varepsilon_{k}u_{k}\in\Sigma^{\ast}$. Here we choose $\varepsilon_{1},\cdots,\varepsilon_{k}\in\Sigma^{M}$, which satisfies for any $1\leq i\leq k$, $\varepsilon_{i}$ is the minimal word (in the sense of lexicographical order) such that $0^{n_{0}}\varepsilon_{1}u_{1}\cdots\varepsilon_{i}u_{i}\in\Sigma^{\ast}$. Let
\begin{equation*}
G_{k}=\{0^{n_{0}}\varepsilon_{1}u_{1}\varepsilon_{2}u_{2}\cdots\varepsilon_{k}u_{k}:u_{i}\in\mathcal{W}_{i},\forall\ 1\leq i\leq k\}.
\end{equation*}
Note that every element of $G_{k}$ is of length $n_{k}$. Define
\begin{equation*}
E_{0}=I_{n_{0}}(0^{n_{0}})
\end{equation*}
and
\begin{equation*}
E_{k}=\bigcup_{w\in G_{k}}I_{n_{k}}(w),\quad \forall\ k\geq1.
\end{equation*}
The desired Cantor-type subset of $\mathcal{U}_{\psi}^{\ast}(a,b)$ is defined as
\begin{equation*}
\mathcal{U}_{P}^{\ast}(a,b):=\bigcap_{k=0}^{\infty}E_{k}.
\end{equation*}
Now, we show that $\mathcal{U}_{P}^{\ast}(a,b)$ is the subset of $\mathcal{U}_{\psi}^{\ast}(a,b)$.
For any $\omega\in\mathcal{U}_{P}^{\ast}(a,b)$, since the elements of $\mathcal{U}_{P}^{\ast}(a,b)$ are of the form $0^{n_{0}}\varepsilon_{1}u_{1}\cdots\varepsilon_{k}u_{k}\cdots$
with $u_{k}\in\mathcal{W}_{k}$ $(k\geq1)$ and $\varepsilon_{k}\in\Sigma^{M}$ is the minimal word (in the sense of lexicographical order) such that $0^{n_{0}}\varepsilon_{1}u_{1}\cdots\varepsilon_{k}u_{k}\in\Sigma^{\ast}$,
by the definition of $\mathcal{W}_{k}$, we have the following two facts:
\begin{enumerate}[(i)]
\item If $n_{k}<N\leq n_{k}+2M$, then
\begin{equation*}
d_{k-1}-1-3M\leq L_{N}(\omega)\leq\max\{d_{k-1}-1+F+M,d_{k}-1-M\}.
\end{equation*}
\item If $n_{k}+2M<N\leq n_{k+1}$, then
\begin{equation*}
d_{k}-1-3M\leq L_{N}(\omega)\leq d_{k}-1+F+M.
\end{equation*}
\end{enumerate}
Combining these and $\eqref{61}$, we obtain that
\begin{equation*}
\liminf_{N\to\infty}\frac{L_{N}(\omega)}{\Phi(N)}\geq\lim_{k\to\infty}\frac{d_{k-1}}{\Phi(n_{k})}=a
\end{equation*}
and
\begin{equation*}
\limsup_{N\to\infty}\frac{L_{N}(\omega)}{\Phi(N)}\leq\lim_{k\to\infty}\frac{d_{k}}{\Phi(n_{k})}=b.
\end{equation*}
Note that
\begin{equation*}
d_{k-1}-1-3M\leq L_{n_{k}+1}(\omega)\leq d_{k-1}-1+F+M,
\end{equation*}
thus
\begin{equation*}
\lim_{k\to\infty}\frac{L_{n_{k}+1}(\omega)}{\Phi(n_{k}+1)}=\lim_{k\to\infty}\frac{d_{k-1}}{\Phi(n_{k})}=a.
\end{equation*}
Furthermore,
\begin{equation*}
\lim_{k\to\infty}\frac{L_{n_{k}+2M+1}(\omega)}{\Phi(n_{k}+2M+1)}=\lim_{k\to\infty}\frac{d_{k}}{\Phi(n_{k})}=b.
\end{equation*}
Therefore, we have $\liminf\limits_{N\to\infty}\frac{L_{N}(\omega)}{\Phi(N)}=a$ and $\limsup\limits_{N\to\infty}\frac{L_{N}(\omega)}{\Phi(N)}=b$.
Thus, $\mathcal{U}_{P}^{\ast}(a,b)\subset\mathcal{U}_{\psi}^{\ast}(a,b)$. Similar with the proof of Lemma $\ref{55}$,
we assign a mass distribution on $\mathcal{U}_{P}^{\ast}(a,b)$. We first assign a mass distribution on all cylinders.
Let $\lambda(I_{k}(0^{k}))=1$ for any $1\leq k\leq n_{0}$. For any $I_{n_{1}}\in E_{1}$, let
\begin{equation*}
\lambda(I_{n_{1}})=\frac{1}{\sharp\mathcal{W}_{1}}\lambda(I_{n_{0}}(0^{n_{0}}))=\frac{1}{\sharp\mathcal{W}_{1}}.
\end{equation*}
Assume that the measure on $I_{n_{k}}\in E_{k}$ $(k\geq1)$ is well defined, then the measure on $I_{n_{k+1}}\in E_{k+1}$ is defined as
\begin{equation*}
\lambda(I_{n_{k+1}})=\frac{1}{\sharp\mathcal{W}_{k+1}}\lambda(I_{n_{k}})=\prod_{j=1}^{k+1}\frac{1}{\sharp\mathcal{W}_{j}}.
\end{equation*}
For any $n_{k}<N<n_{k+1}$, let
\begin{equation*}
\lambda(I_{N})=\sum_{I_{n_{k+1}}\in E_{k+1}:I_{N}\cap I_{n_{k+1}}\neq\emptyset}\lambda(I_{n_{k+1}}).
\end{equation*}
Note that
\begin{equation*}
\lambda(I_{N})=\sum_{I_{N+1}:I_{N+1}\cap I_{N}\neq\emptyset}\lambda(I_{N+1})
\end{equation*}
for all $N\geq1$, thus, $\lambda$ can be extended to a Borel probability measure on $\mathcal{U}_{P}^{\ast}(a,b)$.
Lastly we estimate the Hausdorff dimension of $\mathcal{U}_{P}^{\ast}(a,b)$. We show that
\begin{equation}\label{63}
\dim_{\rm H}\mathcal{U}_{P}^{\ast}(a,b)\geq\left(1-\frac{2}{P}\right)\dim_{\rm H}\Sigma.
\end{equation}
Then we have $\dim_{\rm H}\mathcal{U}_{\psi}^{\ast}(a,b)\geq(1-\frac{2}{P})\dim_{\rm H}\Sigma$ for every $P\geq3$.
Letting $P\to\infty$, we obtain that $\dim_{\rm H}\mathcal{U}_{\psi}^{\ast}(a,b)=\dim_{\rm H}\Sigma$.
To prove $\eqref{63}$, in view of Proposition $\ref{29}$, we only need to prove that for any $\omega\in\mathcal{U}_{P}^{\ast}(a,b)$,
\begin{equation*}
\liminf_{N\to\infty}\frac{\log\lambda(I_{N}(\omega))}{\log|I_{N}(\omega)|}\geq\left(1-\frac{2}{P}\right)\dim_{\rm H}\Sigma,
\end{equation*}
where $I_{N}(\omega)$ is the $N$-th cylinder containing $\omega$. By the definition of $\lambda$, we consider two cases.\\
Case \Rmnum{1}: $n_{k}\leq N\leq n_{k}+d_{k}$. Then
\begin{equation*}
\lambda(I_{N}(\omega))=\lambda(I_{n_{k}}(\omega))=\prod_{j=1}^{k}\frac{1}{\sharp\mathcal{W}_{j}}.
\end{equation*}
By Proposition $\ref{210}$, we have
\begin{equation*}
\begin{aligned}
\liminf_{N\to\infty}\frac{\log \lambda(I_{N}(\omega))}{\log|I_{N}(\omega)|}&=\liminf_{N\to\infty}\frac{-\log\lambda(I_{N}(\omega))}{N\log m}\geq\frac{1}{\log m}\liminf_{k\to\infty}\frac{\sum_{j=1}^{k}\log\sharp\mathcal{W}_{j}}{n_{k}+d_{k}}\\
&=\frac{1}{\log m}\liminf_{k\to\infty}\frac{\sum_{j=1}^{k}\log\sharp\mathcal{W}_{j}}{n_{k}}\geq\frac{1}{\log m}\liminf_{k\to\infty}\frac{\log\sharp\mathcal{W}_{k+1}}{n_{k+1}-n_{k}}.
\end{aligned}
\end{equation*}
Since $\sharp\mathcal{W}_{k+1}=\left(\sharp\Sigma^{d_{k}-2M-1}\right)^{l_{k+1}}$ and $B=\lim\limits_{n\to\infty}\frac{\log\sharp\Sigma^{n}}{n}$, we obtain that
\begin{equation}\label{64}
\lim_{k\to\infty}\frac{\log\sharp\mathcal{W}_{k+1}}{l_{k+1}d_{k}}=B.
\end{equation}
Thus,
\begin{equation*}
\liminf_{N\to\infty}\frac{\log\lambda(I_{N}(\omega))}{\log|I_{N}(\omega)|}\geq\frac{h(\Sigma)}{\log m}\liminf_{k\to\infty}\frac{l_{k+1}d_{k}}{n_{k+1}-n_{k}}.
\end{equation*}
In view of
\begin{equation*}
n_{k+1}-n_{k}=d_{k}(l_{k+1}+1)+r_{k+1}<d_{k}(l_{k+1}+2)
\end{equation*}
and
\begin{equation*}
l_{k+1}:=\left\lfloor\frac{n_{k+1}-n_{k}}{d_{k}}\right\rfloor-1\geq P-2,
\end{equation*}
we obtain that
\begin{equation*}
\liminf_{N\to\infty}\frac{\log\lambda(I_{N}(\omega))}{\log|I_{N}(\omega)|}\geq\frac{h(\Sigma)}{\log m}\cdot\liminf_{k\to\infty}\frac{l_{k+1}}{l_{k+1}+2}\geq\left(1-\frac{2}{P}\right)\frac{h(\Sigma)}{\log m}.
\end{equation*}
Case \Rmnum{2}: $n_{k}+d_{k}<N<n_{k+1}$. Since $n_{k+1}=n_{k}+(l_{k+1}+1)d_{k}+r_{k+1}<n_{k}+(l_{k+1}+2)d_{k}$, there exists $1\leq l\leq l_{k+1}+1$, such that $n_{k}+ld_{k}\leq N<n_{k}+(l+1)d_{k}$. Therefore,
\begin{equation*}
\lambda(I_{N}(\omega))\leq\lambda(I_{n_{k}+ld_{k}}(\omega))=\left(\sharp\Sigma^{d_{k}-1-2M}\right)^{-(l-1)}\lambda(I_{n_{k}}(\omega)).
\end{equation*}
By Proposition $\ref{210}$, we have
\begin{equation*}
\begin{aligned}
\liminf_{N\to\infty}\frac{\log\lambda(I_{N}(\omega))}{\log|I_{N}(\omega)|}&\geq\liminf_{k\to\infty}\frac{-\log\lambda(I_{n_{k}}(\omega))+(l-1)\log\sharp\Sigma^{d_{k}-1-2M}}{ N\log m}\\&\geq\frac{1}{\log m}\liminf_{k\to\infty}\frac{-\log\lambda(I_{n_{k}}(\omega))+(l-1)\log\sharp\Sigma^{d_{k}-1-2M}}{n_{k}+(l+1)d_{k}}\\
&=\frac{1}{\log m}\liminf_{k\to\infty}\frac{-\log\lambda(I_{n_{k}}(\omega))+(l+1)\log\sharp\Sigma^{d_{k}-1-2M}}{n_{k}+(l+1)d_{k}}.
\end{aligned}
\end{equation*}
The above equality follows from $\lim\limits_{k\to\infty}\frac{\log\sharp\Sigma^{d_{k}-1-2M}}{d_{k}}\leq\log m$ and $\lim\limits_{k\to\infty}\frac{d_{k}}{n_{k}}=0$.
Since $B=\lim\limits_{n\to\infty}\frac{\log\sharp\Sigma^{n}}{n}$ and $\lim\limits_{k\to\infty}d_{k}=+\infty$, for any $\delta\in(0,B)$, we have $\log\sharp\Sigma^{d_{k}-1-2M}>(B-\delta)d_{k}$ for all $k$ large enough. Thus,
\begin{equation}\label{65}
\begin{aligned}
\liminf_{N\to\infty}\frac{\log\lambda(I_{N}(\omega))}{\log|I_{N}(\omega)|}&\geq\frac{1}{\log m}\liminf_{k\to\infty}\frac{-\log \lambda(I_{n_{k}}(\omega))+(B-\delta)(l+1)d_{k} }{n_{k}+(l+1)d_{k}}\\ &=\frac{B-\delta}{\log m}\liminf_{k\to\infty}\frac{(B-\delta)^{-1}(-\log \lambda(I_{n_{k}}(\omega)))+(l+1)d_{k}}{n_{k}+(l+1)d_{k}}\\
&\geq\frac{B-\delta}{\log m}\liminf_{k\to\infty}\frac{(B+\delta)^{-1}(-\log\lambda(I_{n_{k}}(\omega)))+(l+1)d_{k}}{n_{k}+(l+1)d_{k}}.
\end{aligned}
\end{equation}
Note that
\begin{equation*}
\begin{aligned}
\limsup_{k\to\infty}\frac{-\log\lambda(I_{n_{k}}(\omega))}{n_{k}}=\limsup_{k\to\infty}\frac{\sum_{j=1}^{k}\log\sharp\mathcal{W}_{j}}{n_{k}}
&\leq\limsup_{k\to\infty}\frac{\log\sharp\mathcal{W}_{k+1}}{n_{k+1}-n_{k}}\\&=\limsup_{k\to\infty}\frac{Bl_{k+1}d_{k}}{n_{k+1}-n_{k}}\leq B,
\end{aligned}
\end{equation*}
the second equality is due to $\eqref{64}$ and the last inequality is due to $n_{k+1}-n_{k}\geq(l_{k+1}+1)d_{k}$. Therefore, for all $k$ large enough, we have
\begin{equation}\label{66}
(B+\delta)^{-1}(-\log\lambda(I_{n_{k}}(\omega)))<n_{k}.
\end{equation}
Combining $\eqref{65}$, $\eqref{66}$ and the inequality $\frac{y+s}{z+s}\geq\frac{y}{z},0\leq y<z,s\geq0$, we obtain that
\begin{equation*}
\liminf_{N\to\infty}\frac{\log\lambda(I_{N}(\omega))}{\log|I_{N}(\omega)|}\geq\frac{B-\delta}{\log m}\liminf_{k\to\infty}\frac{(B+\delta)^{-1}(-\log\lambda(I_{n_{k}}(\omega))) }{n_{k}}.
\end{equation*}
Furthermore, by the proof of Case \Rmnum{1}, we know that
\begin{equation*}
\liminf_{k\to\infty}\frac{-\log\lambda(I_{n_{k}}(\omega))}{n_{k}}\geq\left(1-\frac{2}{P}\right)h(\Sigma).
\end{equation*}
Therefore,
\begin{equation*}
\liminf_{N\to\infty}\frac{\log\lambda(I_{N}(\omega))}{\log|I_{N}(\omega)|}\geq\frac{B-\delta}{B+\delta}\left(1-\frac{2}{P}\right)\frac{ h(\Sigma)}{\log m}.
\end{equation*}
Letting $\delta\to0$, we obtain that
\begin{equation*}
\liminf_{N\to\infty}\frac{\log\lambda(I_{N}(\omega))}{\log|I_{N}(\omega)|}\geq\left(1-\frac{2}{P}\right)\frac{h(\Sigma)}{\log m}.
\end{equation*}
The combination of Case \Rmnum{1} and Case \Rmnum{2} gives
\begin{equation*}
\liminf_{N\to\infty}\frac{\log\lambda(I_{N}(\omega))}{\log|I_{N}(\omega)|}\geq\left(1-\frac{2}{P}\right)\frac{h(\Sigma)}{\log m}=\left(1-\frac{2}{P}\right)\dim_{\rm H}\Sigma.
\end{equation*}
The equality follows from $\dim_{\rm H}\Sigma=\frac{h(\Sigma)}{\log m}$ (see $\rm{\cite[Proposition\ \Rmnum{3}.1]{HF1967}}$). Similar arguments apply to the remaining cases.
We only give the constructions for the sequences $\{n_{k}\}_{k\geq0}$ and $\{d_{k}\}_{k\geq0}$.\\
\textbf{Case 3}: $0<a<b=\infty$. Fix a integer $P\geq3$. We choose $n_{0}$ such that
\begin{equation*}
\left\lfloor a\Phi(n_{0})\log\frac{n_{0}}{\Phi(n_{0})}\right\rfloor\geq P+2+2M
\end{equation*}
and then take $d_{0}=\lfloor a\Phi(n_{0})\log\frac{n_{0}}{\Phi(n_{0})}\rfloor$, for any $k\geq1$, $\{n_{k}\}$ and $\{d_{k}\}$ are recursively defined as
\begin{equation*}
n_{k}:=\left\lfloor\Phi^{-1}\left(\Phi(n_{k-1})\log\frac{n_{k-1}}{\Phi(n_{k-1})}\right)\right\rfloor+Pd_{k-1}\ {\rm and}\ d_{k}:=\left\lfloor a\Phi(n_{k})\log\frac{n_{k}}{\Phi(n_{k})}\right\rfloor.
\end{equation*}
\textbf{Case 4}: $0=a<b<\infty$. Fix a integer $P\geq3$. We choose $n_{0}$ such that
\begin{equation*}
\lfloor b\Phi(n_{0})\rfloor\geq P+2+2M
\end{equation*}
and then take $d_{0}=\lfloor b\Phi(n_{0})\rfloor$, for any $k\geq1$, $\{n_{k}\}$ and $\{d_{k}\}$ are recursively defined as
\begin{equation*}
n_{k}:=\lfloor\Phi^{-1}((k-1)\Phi(n_{k-1}))\rfloor+Pd_{k-1}\ {\rm and}\ d_{k}:=\lfloor b\Phi(n_{k})\rfloor.
\end{equation*}
\textbf{Case 5}: $0=a<b=\infty$. Fix a integer $P\geq3$. We choose $n_{0}$ such that
\begin{equation*}
\lfloor\sqrt{n_{0}\Phi(n_{0})}\rfloor\geq P+2+2M
\end{equation*}
and then take $d_{0}=\lfloor\sqrt{n_{0}\Phi(n_{0})}\rfloor$, for any $k\geq1$, $\{n_{k}\}$ and $\{d_{k}\}$ are recursively defined as
\begin{equation*}
n_{k}:=\lfloor\Phi^{-1}(n_{k-1}\Phi(n_{k-1}))\rfloor+Pd_{k-1}\ {\rm and}\ d_{k}:=\lfloor\sqrt{n_{k}\Phi(n_{k})}\rfloor.
\end{equation*}
\textbf{Case 6}: $a=b=0$. Fix a integer $P\geq3$. We choose $n_{0}$ such that
\begin{equation*}
\lfloor\sqrt{\Phi(n_{0})}\rfloor\geq P+2+2M
\end{equation*}
and then take $d_{0}=\lfloor\sqrt{\Phi(n_{0})}\rfloor$, for any $k\geq1$, $\{n_{k}\}$ and $\{d_{k}\}$ are recursively defined as
\begin{equation*}
n_{k}:=n_{k-1}+Pd_{k-1}+1\ {\rm and}\ d_{k}:=\lfloor\sqrt{\Phi(n_{k})}\rfloor.
\end{equation*}
\textrm{}
\end{proof}
\subsection{Proof of Theorem $\ref{113}$ Cases \rmnum{1}\ and \rmnum{2}}
\begin{proof}[\rm \textbf{Proof}]
Case \rmnum{1}\quad Since $0<\tau<+\infty$, we have
\begin{equation}\label{612}
\limsup\limits_{N\to\infty}\frac{L_{N}(\omega)}{-\log_{m}\psi(N)}=\tau^{-1}\overline{L}(\omega)\ {\rm and}\ \liminf\limits_{N\to\infty}\frac{L_{N}(\omega)}{-\log_{m}\psi(N)}=\tau^{-1}\underline{L}(\omega).
\end{equation}
Thus,
\begin{equation}\label{68}
\mathcal{U}_{\psi}(a,b)=\mathcal{U}_{\psi_{0}}(\tau a,\tau b)
\end{equation}
and
\begin{equation}\label{69}
\mathcal{U}_{\psi}^{\ast}(a,b)=\mathcal{U}_{\psi_{0}}^{\ast}(\tau a,\tau b).
\end{equation}
Recall that $\psi_{0}(N)=m^{-N}$. In view of $\eqref{612}$, $\eqref{68}$, $\eqref{69}$ and Lemma $\ref{55}$, we immediately obtain (2) of Theorem $\ref{113}$ Case \rmnum{1}. \\
The remaining case is $a\geq\tau^{-1}$. If $b<+\infty$, the upshot of  $\eqref{41}$, $\eqref{68}$ and Proposition $\ref{211}$ is that $\mathcal{U}_{\psi}(a,b)=\emptyset$. Therefore,
$\mathcal{U}_{\psi}^{\ast}(a,b)=\mathcal{U}_{\psi}(a,b)=\emptyset$. \\ If $b=+\infty$, in view of $\eqref{68}$, we have
$\mathcal{U}_{\psi}(a,b)\subset\{\omega\in\Sigma:\overline{L}(\omega)=+\infty\}$. Since $\dim_{\rm H}(\{\omega\in\Sigma:\overline{L}(\omega)=+\infty\})=0$, we obtain that $\dim_{\rm H}\mathcal{U}_{\psi}(a,b)=0$.
Therefore, $\dim_{\rm H}\mathcal{U}_{\psi}^{\ast}(a,b)=\dim_{\rm H}\mathcal{U}_{\psi}(a,b)=0$.\\ What is more, by Proposition $\ref{211}$, we know that $\underline{L}(\omega)>1$ is equivalent to $\underline{L}(\omega)=+\infty$.
Combining this and $\eqref{69}$, we have $\mathcal{U}_{\psi}^{\ast}(a,b)=\emptyset$ when $\tau^{-1}<a<+\infty$. Lastly, if $a>\tau^{-1}$
and $b=+\infty$, in view of $\eqref{68}$ and Proposition $\ref{211}$, we obtain that $\mathcal{U}_{\psi}(a,b)=\bigcup_{n=0}^{\infty}\sigma^{-n}(\{0^{\infty}\})$.\\
Case \rmnum{2}\quad Since $\tau=+\infty$, we have
\begin{equation}\label{610}
\limsup\limits_{N\to\infty}\frac{L_{N}(\omega)}{-\log_{m}\psi(N)}>0\ {\rm implies}\ \overline{L}(\omega)=+\infty.
\end{equation}
and
\begin{equation}\label{611}
\liminf\limits_{N\to\infty}\frac{L_{N}(\omega)}{-\log_{m}\psi(N)}>0\ {\rm implies}\ \underline{L}(\omega)=+\infty.
\end{equation}
If $b>0$, it follows from $\eqref{610}$ that $\mathcal{U}_{\psi}(a,b)\subset\{\omega\in\Sigma:\overline{L}(\omega)=+\infty\}$.
Since $\dim_{\rm H}(\{\omega\in\Sigma:\overline{L}(\omega)=+\infty\})=0$, we obtain that $\dim_{\rm H}\mathcal{U}_{\psi}(a,b)=0$. Hence,
$$\dim_{\rm H}\mathcal{U}_{\psi}^{\ast}(a,b)=\dim_{\rm H}\mathcal{U}_{\psi}(a,b)=0.$$
The remaining case is $b=0$, more precisely $a=b=0$ (we assume that $0\leq a\leq b$). Note that $\mathcal{U}_{\psi_{0}}^{\ast}(a',b')\subset\mathcal{U}^{\ast}_{\psi}(0,0)$ for all $0\leq a'<1$ and $\frac{a'}{1-a'}\leq b'<+\infty$,
it follows from Proposition $\ref{55}$ that $\dim_{\rm H}\mathcal{U}^{\ast}_{\psi}(0,0)\geq\frac{b'-b'a'-a'}{(1+b')(b'-a')}\dim_{\rm H}\Sigma$ for all $0\leq a'<1$ and $\frac{a'}{1-a'}\leq b'<+\infty$.
Thus, $\dim_{\rm H}\mathcal{U}^{\ast}_{\psi}(0,0)=\dim_{\rm H}\Sigma.$ Furthermore, when $a>0$, it follows from $\eqref{611}$ and Proposition $\ref{211}$ that
$$\mathcal{U}_{\psi}(a,b)\subset\bigcup_{n=0}^{\infty}\sigma^{-n}(\{0^{\infty}\}).$$
\textrm{}
\end{proof}
\subsection{Proof of Corollary $\ref{17}$}
\begin{proof}[\rm \textbf{Proof}]
Recall that $\psi_{0}(N)=m^{-N}$, thus $\tau=\lim\limits_{N\to\infty}\frac{-\log_{m}\psi_{0}(N)}{N}=1$.\\
(\rmnum{1}) If $a>1$, by Proposition $\ref{211}$, we have $\mathcal{U}_{\psi_{0}}(a)=\bigcup\limits_{n=0}^{\infty}\sigma^{-n}(\{0^{\infty}\})$.\\
(\rmnum{2}) Since $\mathcal{U}_{\psi_{0}}(0)=\Sigma$, we only need to consider $0<a\leq1$. If $0<a<1$, it follows from Lemma $\ref{55}$ (\rmnum{2}) that
$$\dim_{\rm H}\mathcal{U}_{\psi_{0}}(a)\geq\frac{b(1-a)-a}{(1+b)(b-a)}\dim_{\rm H}\Sigma,\ \forall\ b\geq\frac{a}{1-a}.$$
An easy calculation shows that the maximum of the right hand side of above inequality is attained for $b=\frac{2a}{1-a}$.
Thus, $\dim_{\rm H}\mathcal{U}_{\psi_{0}}(a)\geq\left(\frac{1-a}{1+a}\right)^{2}\dim_{\rm H}\Sigma$. On the other hand, for any $\varrho>0$, in view of
Lemma $\ref{55}$ (\rmnum{1}) and the denseness of the set of rational numbers, we have that
\begin{equation*}
\begin{aligned}
&\quad\mathcal{U}_{\psi_{0}}(a)\cap\{\omega\in\Sigma:\overline{L}(\omega)<+\infty\}\\
&\subset\mathcal{U}_{\psi_{0}}(a,\frac{a}{1-a})\cup\bigcup_{b\in[\frac{a}{1-a},+\infty)\cap\mathbb{Q}}\{\omega\in\Sigma:\underline{L}(\omega)\geq a,b\leq\overline{L}(\omega)\leq b+\varrho\}.
\end{aligned}
\end{equation*}
By Lemma $\ref{55}$ (\rmnum{2}) and the countable stability of Hausdorff dimension, we obtain that
$$\dim_{\rm H}\mathcal{U}_{\psi_{0}}(a)\leq\sup_{b\in[\frac{a}{1-a},+\infty)}\left(\frac{b(1-a)-a}{(1+b)(b-a)}+\frac{\varrho(1-a)^{2}}{a^{3}}\right)\dim_{\rm H}\Sigma.$$
Similar to the above proof, the maximum of the right hand side of above inequality is attained for $b=\frac{2a}{1-a}$. Therefore,
$\dim_{\rm H}\mathcal{U}_{\psi_{0}}(a)\leq\left(\frac{1-a}{1+a}\right)^{2}\dim_{\rm H}\Sigma$. Therefore,
$$\dim_{\rm H}\mathcal{U}_{\psi_{0}}(a)=\left(\frac{1-a}{1+a}\right)^{2}\dim_{\rm H}\Sigma.$$ The remaining case is $a=1$, it follows from Propositions
$\ref{211}$ and $\ref{51}$ that $\mathcal{U}_{\psi_{0}}(1)=\mathcal{U}_{\psi_{0}}(1,+\infty)$. An easy cover argument shows that
$\dim_{\rm H}(\{\omega\in\Sigma:\overline{L}(\omega)=+\infty\})=0$, hence $\dim_{\rm H}\mathcal{U}_{\psi_{0}}(1)=0$.
\textrm{}
\end{proof}
\subsection*{Acknowledgements}
It was supported by NSFC12271176 and Guangdong Natural Science Foundation 2023A1515010691.

\end{document}